\documentclass[12pt]{article}
\usepackage[]{amsmath,amssymb}
\usepackage{amscd}
\usepackage{latexsym}
\usepackage{cite}
\usepackage{amsthm}

\newtheorem{definition}{Definition}[section]

\newtheorem{lemma}[definition]{Lemma}

\newtheorem{remark}[definition]{Remark}

\newtheorem{note}[definition]{Note}

\newtheorem{proposition}[definition]{Proposition}

\begin{document}
\title{\bf 
The 
alternating central extension \\
 of the Onsager Lie algebra}
\author{
Paul Terwilliger 
}
\date{}
%\footnote{This author gratefully acknowledges 
%support from the FY2007 JSPS Invitation Fellowship Program
%for Reseach in Japan (Long-Term), grant L-07512.}
%}
%\date{}
%to get date printout, comment out above line

\maketitle
\begin{abstract}
The Onsager Lie algebra $O$ is often used to study integrable lattice models. The universal enveloping algebra of $O$
admits a $q$-deformation $O_q$ called the 
$q$-Onsager algebra. Recently, an algebra $\mathcal O_q$ was introduced called the alternating central extension of $O_q$.
In this paper we introduce a Lie algebra $\mathcal O$ that is roughly described by the following two analogies: (i) $\mathcal O$ is to $O$ as $\mathcal O_q$ is to $O_q$;
(ii) $O_q$ is to $O$ as $\mathcal O_q$ is to $\mathcal O$. We call $\mathcal O$ the alternating central extension of $O$. This paper
contains a comprehensive description of $\mathcal O$.
\bigskip


 \end{abstract}
 
 \section{Introduction} 
 
In the seminal 1944 paper 
\cite{Onsager},
 Lars Onsager obtained the free energy of
the two-dimensional Ising model in a zero magnetic field. In that paper Onsager introduced 
 an infinite-dimensional Lie algebra, now called the Onsager Lie algebra.
We will denote the Onsager Lie algebra by $O$. 
\medskip

\noindent There are various ways to describe $O$. For instance,
Onsager defined $O$ by giving a linear basis and
the action of the Lie bracket on the basis; see Lemma \ref{def:ons}
below.
 In
\cite{perk},
Perk gave
a presentation of $O$ by generators and relations. This 
presentation involves two generators and two relations, called the
Dolan/Grady relations
\cite{Dolgra}; see Lemma  \ref{def:OA} below and
\cite[Remark~9.1]{madrid}.  In \cite{roan}, Roan embedded $O$ into the loop algebra  $L(\mathfrak{sl}_2)$. From this point of view, $O$ consists
of the elements in $L(\mathfrak{sl}_2)$ that are fixed by a certain involution; see Definition  \ref{def:onsager}
below.
\medskip

\noindent The Lie algebra $O$ has been investigated many times in connection with integrable lattice models; see  \cite{ly, Miao, shibata, vernier} and the references therein.
It is shown in \cite{bbc2018} that 
 $O$ admits a presentation \`a la Faddeev-Reshetkhin-Taktadjan (FRT). This presentation is used in
\cite[Definition~4.1 and Theorem~2]{BC17} to obtain
 an attractive basis for $O$; see Lemma \ref{lem:OA2} below.
The finite-dimensional irreducible $O$-modules are classified up to isomorphism \cite{Davfirst, Da, dRC}; see \cite[Section~1]{Ha} for a summary.
In  \cite{neher}, the finite-dimensional irreducible $O$-modules are described from the viewpoint of equivariant map algebras.
The closed ideals of $O$ are obtained in \cite{dateRoan, dRC}, and the ideals of $O$ are obtained in \cite{E}. 
In \cite{HT} the tetrahedron Lie algebra $\boxtimes$ is introduced, and found to be a direct sum of three Lie subalgebras, each isomorphic to $O$. In \cite{Ha, 3pt} the representation theories of $O$ and $\boxtimes$ are compared.
%%%%In \cite[Definition~4.1 and Theorem~2]{BC17}
%%%% an attractive basis for $O$ is obtained using an FRT presentation; see Lemma \ref{lem:OA2} below.
 A comprehensive summary of $O$ can be found in \cite{chaar}.
%%See \cite{dateRoan, dRC,Davfirst, Da, chaar} and \cite{Ha, HT, 3pt} for mathematical treatments of $O$.  See \cite{Miao} for a recent summary of $O$ with a physics emphasis.
% A comprehensive summary of $O$ can be found in \cite{chaar}.
\medskip

\noindent In the present paper we introduce a Lie algebra $\mathcal O$, called the alternating central extension of $O$. Our motivation
for $\mathcal O$ is a bit intricate, and involves four other algebras:  the $q$-Onsager algebra $O_q$, its alternating central extension $\mathcal O_q$, 
the positive part $U^+_q$ of the quantized enveloping algebra $U_q(\widehat{\mathfrak{sl}}_2)$,  and its alternating central extension $\mathcal U^+_q$.
Our motivation for $\mathcal O$ is explained in the paragraphs below.
\medskip

\noindent We mentioned that $O$ has a presentation involving two generators and  the Dolan/Grady relations.
Via this presentation,
the universal enveloping algebra of $O$
admits a $q$-deformation $O_q$ called the 
$q$-Onsager algebra
 \cite{bas1},
\cite{qSerre}. The algebra $O_q$ is associative and 
infinite-dimensional.
 It is defined by two generators and two relations
 called the $q$-Dolan/Grady relations; see Definition \ref{def:U}
 below.
The $q$-Dolan/Grady relations first appeared in algebraic combinatorics,
in the study of $Q$-polynomial distance-regular graphs 
\cite[Lemma~5.4]{tersub3}. Shortly thereafter they appeared in mathematical physics
\cite[Section~2]{bas1}.
Up to the present, 
$O_q$ remains an active  research topic in
mathematics 
\cite{BK, ito, ItoTerAug, LRW, qSerre,
madrid,
uaw,
lusztigaut,
pospart,
pbw,
z2z2z2,
diagram,
conj,
pbwqO,
compQons}
and physics
\cite{bas2,
bas1, 
BK05,
bas4,
basKoi,
bas8,
basXXZ,
nonabel2017,
basBel,
basnc,
BVu
}.
\medskip

\noindent 
\noindent
 In \cite{BK05} Baseilhac and Koizumi introduced a current algebra  for $O_q$, in order to solve boundary integrable systems with hidden symmetries. We will denote this current algebra by $\mathcal O_q$.
In \cite[Definition~3.1]{basnc} Baseilhac and Shigechi gave a presentation of $\mathcal O_q$ by generators and relations. The generators are denoted
\begin{align}
 \lbrace \mathcal W_{-k}\rbrace_{k\in \mathbb N}, \qquad \lbrace \mathcal W_{k+1}\rbrace_{k\in \mathbb N}, \qquad  
 \lbrace \mathcal G_{k+1} \rbrace_{k\in \mathbb N},
\qquad
\lbrace  \mathcal {\tilde G}_{k+1} \rbrace_{k\in \mathbb N}
\label{eq:aalt}
\end{align}
and the relations are given in 
\eqref{eq:3p1a}--\eqref{eq:3p11a} below.
These relations are a bit complicated, and it is not transparently clear how $O_q$,  $\mathcal O_q$ are related.
Some conjectures were made in
 \cite[Conjecture~2]{basBel}, 
 \cite[Conjectures~4.6,~4.8]{z2z2z2} and eventually resolved  \cite[Section~1]{pbwqO}, \cite[Theorem~1.1]{compQons}.
  \medskip
 
 \noindent 
Hoping to shed light on how  $O_q$, $\mathcal O_q$ are related, in \cite{alternating, altCE}
we considered an analogous but simpler situation, in which $O_q$, $\mathcal O_q$  are replaced by $U^+_q$, $\mathcal U^+_q$.
The algebra $U^+_q$ is defined by
two generators $A$, $B$ and two relations called the $q$-Serre relations;
see \cite[Definition~2.2]{alternating}. The algebra $\mathcal U^+_q$ will be described shortly.
In \cite[Section~5]{alternating} 
we introduced the alternating generators for $U^+_q$.
There are four types of
alternating generators,
denoted
\begin{align*}
\lbrace  W_{-k}\rbrace_{k\in \mathbb N}, \quad
\lbrace  W_{k+1}\rbrace_{k\in \mathbb N}, \quad
\lbrace  G_{k+1}\rbrace_{k\in \mathbb N}, \quad
\lbrace {\tilde G}_{k+1}\rbrace_{k\in \mathbb N}.
\end{align*}
%By 
%\cite[Lemma 5.11]{alternating} the alternating elements of
%each type mutually commute.
%\medskip

\noindent 
The alternating generators get their name in the following way.
Start with the free algebra $\mathbb V$ on two generators $x,y$.
The standard  basis for
$\mathbb V$ consists of the words in $x,y$.
In
\cite{rosso1, rosso} M. Rosso introduced
an algebra structure on $\mathbb V$, called a
$q$-shuffle algebra.
For $u,v\in \lbrace x,y\rbrace$ their
$q$-shuffle product is
$u\star v = uv+q^{\langle u,v\rangle }vu$, where
$\langle u,v\rangle =2$
(resp. $\langle u,v\rangle =-2$)
if $u=v$ (resp.
 $u\not=v$).
Rosso gave an injective algebra homomorphism 
from $U^+_q$ into the $q$-shuffle algebra
${\mathbb V}$, that sends $A\mapsto x$ and $B\mapsto y$.
By
\cite[Definition~5.2]{alternating}
the homomorphism sends
\begin{align*}
&W_0 \mapsto x, \qquad W_{-1} \mapsto xyx, \qquad W_{-2} \mapsto xyxyx, \qquad \ldots
%\label{eq:WmIntro}
\\
&W_1 \mapsto y, \qquad W_{2} \mapsto yxy, \qquad W_{3} \mapsto yxyxy, \qquad \ldots
%\label{eq:WpIntro}
\\
&G_{1} \mapsto yx, \qquad G_{2} \mapsto yxyx,  \qquad G_3 \mapsto yxyxyx, \qquad \ldots
%\label{eq:GIntro}
\\
&\tilde G_{1} \mapsto xy, \qquad \tilde G_{2} \mapsto 
xyxy,\qquad \tilde G_3 \mapsto xyxyxy, \qquad \ldots
%\label{eq:GtIntro}
\end{align*}
In \cite{alternating} we used the above homomorphism
to obtain many relations involving the 
 alternating generators; see \cite[Propositions~5.7,~5.10]{alternating}.
These relations can be viewed as a limiting case of \eqref{eq:3p1a}--\eqref{eq:3p11a}, in which the parameter $\rho$ below \eqref{eq:3p11a}  goes to zero.
\medskip

 \noindent In \cite[Definition~3.1]{altCE}
we defined the algebra $\mathcal U^+_q$ by generators and relations in the following way. The generators, said to be alternating,
are in bijection with the alternating generators of $U^+_q$. The relations are the ones in \cite[Propositions~5.7, 5.10]{alternating}. By construction there
exists a surjective algebra homomorphism  $\mathcal U^+_q\to U^+_q$ that sends each alternating generator of $\mathcal U^+_q$
to the corresponding alternating generator of $U^+_q$. In \cite[Lemma~3.6, Theorem~5.17]{altCE} we adjusted this homomorphism to get an algebra isomorphism
$\mathcal U^+_q \to U^+_q \otimes \mathbb F \lbrack z_1, z_2, \ldots \rbrack$, where $\mathbb F$ is the ground field and
$\lbrace z_n\rbrace_{n=1}^\infty $ are mutually commuting indeterminates. We also showed in \cite[Theorem~10.2]{altCE} that the alternating generators  form a PBW basis
for $\mathcal U^+_q$. 
\medskip

\noindent Using the example of $U^+_q$, $\mathcal U^+_q$ as a guide, 
in \cite{pbwqO} we considered how $O_q$, $\mathcal O_q$ are related.  Employing
the Bergman diamond lemma, we were able to show in  \cite[Theorem~6.1]{pbwqO}
 that the generators \eqref{eq:aalt}
form a PBW basis for $\mathcal O_q$. We used this in \cite[Theorem~9.14]{pbwqO} to establish an algebra isomorphism $\mathcal O_q \to O_q \otimes \mathbb F \lbrack z_1, z_2, \ldots \rbrack$.
\medskip

\noindent We just described how $O_q$, $\mathcal O_q$ are related. We now return to the Lie algebras $O$, $\mathcal O$.
We can obtain $O$ from $O_q$ as a $q=1$ limit. As we will see, we can similarly obtain $\mathcal O$ from $\mathcal O_q$ as a $q=1$ limit. Changing our point of view,
we can say that $\mathcal O_q$ is a $q$-analog of $\mathcal O$, in the same way that $O_q$ is a $q$-analog of $O$.
 To our knowledge $\mathcal O$ has not been considered previously. 
 \medskip
 
 \noindent
 In this paper we give a comprehensive description of $\mathcal O$. Our results are summarized as follows.
 We construct an ideal $\mathcal L$ of the loop algebra $L(\mathfrak{gl}_2)$ that has codimension 1. 
 We define $\mathcal O$ to be the set of elements in $\mathcal L$ that are fixed by a certain involution.
 We display two bases for $\mathcal O$, and give the action of the Lie bracket on each basis.
 We give three presentations of $\mathcal O$ by generators and relations. We show that the center $\mathcal Z$ of $\mathcal O$ has infinite dimension, and we display two bases for $\mathcal Z$. 
  We obtain the direct sum $\mathcal O=O+\mathcal Z$. We display
an automorphism $\sigma$ of $\mathcal O$ and an antiautomorphism $\dagger$ of $\mathcal O$. We show how $\sigma$ and $\dagger$ act on the above bases for $\mathcal O$. We
show that $\sigma$ and $\dagger$ fix everything in $\mathcal Z$.
After a brief description of $O_q$ and $\mathcal O_q$, we explain the $q=1$ limit process that
allows us to recover $O$ from $O_q$ and  $\mathcal O$ from $\mathcal O_q$. At the end of the paper, we make many comparisons between $O$, $\mathcal O$ and $O_q$, $\mathcal O_q$.
\medskip

\noindent In condensed form, our main results are roughly summarized by the following two analogies: (i) $\mathcal O$ is to $O$ as $\mathcal O_q$ is to $O_q$;
(ii) $O_q$ is to $O$ as $\mathcal O_q$ is to $\mathcal O$. 
\medskip

\noindent This paper is organized as follows. In Section 2 we give some preliminaries, that are mainly about the loop algebras $L(\mathfrak{sl}_2)$ and $L(\mathfrak{gl}_2)$. In Section 3, 
we review some basic facts about $O$. In Section 4, we give our comprehensive description of $\mathcal O$. In Section 5, we describe  $O_q$, $\mathcal O_q$ and their relationship to $O$, $\mathcal O$.

 \section{Preliminaries} We now begin our formal argument. Throughout the paper, the following notational conventions are in effect. Recall the natural numbers $\mathbb N = \lbrace 0,1,2,\ldots \rbrace$ and
 integers $\mathbb Z = \lbrace 0, \pm 1, \pm 2, \ldots \rbrace$. Define $\mathbb N^+ = \mathbb N\backslash \lbrace 0 \rbrace$.
 Let $\mathbb F$ denote a field with characteristic zero. Every vector space and tensor product mentioned in this paper is over $\mathbb F$. We will be discussing associative algebras and Lie algebras.
  Every algebra with no Lie prefix mentioned in this paper, is understood to be associative and have a multiplicative identity. A subalgebra  has the same multiplicative identity as the parent algebra.
 Every algebra and Lie algebra mentioned in  this paper is over $\mathbb F$.
 %Let $\mathcal A$ denote an algebra. A subalgebra of $\mathcal A$ has the same multiplicative identity as $\mathcal A$. 
  %%Recall that $\mathcal A$ is commutative whenever $ab=ba$ for all $a,b\in \mathcal A$.
  %%%An element $a \in \mathcal A$ is called {\it central} whenever $ab=ba$ for all $b \in \mathcal A$. The set of central elements in $\mathcal A$ forms a subalgebra of $\mathcal A$, called the {\it center} of $\mathcal A$.
  %%The center of $\mathcal A$ is commutative.
  \medskip
  
  \noindent We now introduce an algebra that will play a role in our discussion.
Let $t$ denote an indeterminate. Let $\mathbb F\lbrack t,t^{-1} \rbrack$ denote the algebra of Laurent polynomials in $t, t^{-1}$ that have all coefficients in $\mathbb F$.
The vector space $\mathbb F\lbrack t,t^{-1} \rbrack$ has a basis $\lbrace t^n \rbrace_{n \in \mathbb Z}$.  The algebra $\mathbb F\lbrack t, t^{-1} \rbrack$ has an automorphism
$\tau$ that sends $t\mapsto t^{-1}$.
 \begin{definition}\rm 
(See \cite[p.~299]{damiani}.)
Let $ \mathcal A$ denote an algebra. A {\it Poincar\'e-Birkhoff-Witt} (or {\it PBW}) basis for $\mathcal A$
consists of a subset $\Omega \subseteq \mathcal A$ and a linear order $<$ on $\Omega$
such that the following is a basis for the vector space $\mathcal A$:
\begin{align*}
a_1 a_2 \cdots a_n \qquad n \in \mathbb N, \qquad a_1, a_2, \ldots, a_n \in \Omega, \qquad
a_1 \leq a_2 \leq \cdots \leq a_n.
\end{align*}
We interpret the empty product as the multiplicative identity in $\mathcal A$.
\end{definition}
  \noindent Next, we have some general comments about Lie algebras. An introduction to this topic can be found in \cite{carter}.
  Let $\mathfrak{g}$ denote a Lie algebra with Lie bracket $\lbrack\,,\,\rbrack$. Recall that $\mathfrak{g}$ is abelian
  whenever $\lbrack u,v\rbrack=0$ for all $u,v \in \mathfrak{g}$.  An element $u \in \mathfrak{g}$ is called {\it central} whenever $\lbrack u,v\rbrack=0$ for all $v \in \mathfrak{g}$.
 The set of central elements in $\mathfrak{g}$ forms an ideal of $\mathfrak{g}$, called the {\it center} of $\mathfrak{g}$. The center of $\mathfrak{g} $ is abelian.
 Let ${\rm Mat}_2(\mathbb F)$ denote the algebra consisting  of the $2 \times 2$ matrices that have all entries in $\mathbb F$. Let $I$ denote the identity matrix in ${\rm Mat}_2(\mathbb F)$.
 Let $\mathfrak{gl}_2$ denote the Lie algebra  consisting of the vector space ${\rm Mat}_2(\mathbb F)$ and Lie bracket $\lbrack u,v \rbrack = uv-vu$.
 The center of $\mathfrak{gl}_2$ is equal to $\mathbb F I$.  View the vector space $\mathbb F$ as a Lie algebra with Lie bracket
$\lbrack u,v \rbrack=0$. The trace map ${\rm tr}: \mathfrak{gl}_2 \to \mathbb F$ is a Lie algebra homomorphism, with kernel denoted by
$\mathfrak{sl}_2$. The sum $\mathfrak{gl}_2 = \mathbb F I + \mathfrak{sl}_2$ is direct.
 The vector space $\mathfrak{sl}_2$ has a basis
 \begin{align*}
e = \begin{pmatrix} 0 & 1 \\ 0 & 0 \end{pmatrix},  \qquad  
f = \begin{pmatrix}0 & 0 \\ 1 & 0 \end{pmatrix}, \qquad
h = \begin{pmatrix}1 &\ \ 0 \\ 0 & -1 \end{pmatrix}
\end{align*}
and 
\begin{align*}
\lbrack h,e\rbrack = 2e, \qquad 
\lbrack h,f\rbrack = -2f, \qquad
\lbrack e,f \rbrack = h.
\end{align*}
For a Lie algebra $\mathfrak{g}$, let
$L(\mathfrak{g})$ denote the Lie algebra 
consisting of the vector space
$\mathfrak{g} \otimes \mathbb F \lbrack t,t^{-1}\rbrack$
and Lie bracket
\begin{eqnarray*}
\lbrack u\otimes \varphi, v \otimes \phi \rbrack = \lbrack u,v\rbrack \otimes \varphi \phi,
\qquad \qquad 
u,v \in \mathfrak{g},  
\qquad \varphi, \phi \in  
 \mathbb F \lbrack t, t^{-1}\rbrack.
\end{eqnarray*}
We call $L(\mathfrak{g})$ the {\it loop algebra} for $\mathfrak{g}$.
For $\mathfrak{g}=\mathfrak{gl}_2$ or $\mathfrak{g}=\mathfrak{sl}_2$, we sometimes view an element $x \in L(\mathfrak{g})$
as a $2 \times 2$ matrix that has entries in $\mathbb F \lbrack t,t^{-1}\rbrack$:
\begin{align*}
x = \begin{pmatrix} a(t) &b(t) \\ c(t) & d(t) \end{pmatrix},\qquad \qquad a(t), \;b(t), \;c(t),\; d(t) \;\in\; \mathbb F\lbrack t, t^{-1} \rbrack.
\end{align*}

\noindent Next, we describe an  automorphism $\theta$ of $L(\mathfrak{gl}_2)$. Define the matrix $z=\bigl(\begin{smallmatrix}0 & 1 \\ 1 & 0 \end{smallmatrix}\bigr)$ and note that $z^2 = I$.
The Lie algebra $\mathfrak{gl}_2$ has an automorphism that sends $u \mapsto z u z^{-1}$ for all $u \in \mathfrak{gl}_2$. For
$u=\bigl(\begin{smallmatrix}a &b \\ c & d \end{smallmatrix}\bigr)$ we have $zuz^{-1}=\bigl(\begin{smallmatrix}d & c \\ b & a \end{smallmatrix}\bigr)$.
The automorphism $\theta$ of  $L(\mathfrak{gl}_2)$ sends $u\otimes \varphi \mapsto zuz^{-1} \otimes \tau(\varphi)$
for all $u \in \mathfrak{gl}_2$ and $\varphi \in \mathbb F\lbrack t, t^{-1} \rbrack$. For $x \in L(\mathfrak{gl}_2)$ the elements $x$, $\theta(x)$ are related as follows:
\begin{align}
x = \begin{pmatrix} a(t) & b(t) \\ c(t) & d(t) \end{pmatrix}, \qquad \qquad
\theta(x) = \begin{pmatrix} d(t^{-1}) & c(t^{-1}) \\ b(t^{-1}) & a(t^{-1}) \end{pmatrix}.
\label{eq:xtheta}
\end{align} 
\noindent The automorphism $\theta$ sends
\begin{align*}
&I \otimes t^n \mapsto I \otimes t^{-n}, \qquad \qquad
e\otimes t^n \mapsto f \otimes t^{-n}, \\
&f\otimes t^n \mapsto e \otimes t^{-n}, \qquad \qquad
h \otimes t^n \mapsto -h \otimes t^{-n}
\end{align*}
for $n \in \mathbb Z$. We have $\theta^2 = {\rm id}$, where id denotes the identity map.
The automorphism $\theta$ leaves $L(\mathfrak{sl}_2)$ invariant, and the restriction of $\theta$ to $L(\mathfrak{sl}_2)$ is an automorphism
of $L(\mathfrak{sl}_2)$.
\medskip

\noindent  
Next, we describe an automorphism $\sigma$ of $L(\mathfrak{gl}_2)$. Define the matrix $\zeta=\bigl(\begin{smallmatrix}0 & t \\ 1 & 0 \end{smallmatrix}\bigr)$ and note that $\zeta^2 = t I$.
%%The Lie algebra $\mathfrak{gl}_2$ has an automorphism that sends $u \mapsto \zeta u \zeta^{-1}$ for all $u \in \mathfrak{gl}_2$.
%%% For $u=\bigl(\begin{smallmatrix}a &b \\ c & d \end{smallmatrix}\bigr)\in \mathfrak{gl}_2$ we have $\zeta u\zeta^{-1}=\bigl(\begin{smallmatrix}d & ct \\ bt^{-1} & a \end{smallmatrix}\bigr)$.
 The automorphism $\sigma$ of  $L(\mathfrak{gl}_2)$ sends $x \mapsto \zeta x \zeta^{-1}$  
 %%%$u\otimes \varphi \mapsto \zeta u \zeta^{-1} \otimes \varphi $
for all $x \in L(\mathfrak{gl}_2)$. %%%and $\varphi \in \mathbb F\lbrack t, t^{-1} \rbrack$.
 For $x \in L(\mathfrak{gl}_2)$ the elements $x$, $\sigma(x)$ are related as follows:
\begin{align}
x = \begin{pmatrix} a(t) & b(t) \\ c(t) & d(t) \end{pmatrix}, \qquad \qquad
\sigma(x) = \begin{pmatrix} d(t) & c(t)t \\ b(t)t^{-1} & a(t) \end{pmatrix}.
\label{eq:xsigma}
\end{align} 
\noindent The automorphism $\sigma$ sends
\begin{align*}
&I \otimes t^n \mapsto I \otimes t^{n}, \qquad \qquad
e\otimes t^n \mapsto f \otimes t^{n-1}, \\
&f\otimes t^n \mapsto e \otimes t^{n+1}, \qquad \qquad
h \otimes t^n \mapsto -h \otimes t^{n}
\end{align*}
for $n \in \mathbb Z$. We have $\sigma^2 = {\rm id}$.
The automorphism $\sigma$ leaves $L(\mathfrak{sl}_2)$ invariant, and the restriction of $\sigma$ to $L(\mathfrak{sl}_2)$ is an automorphism
of $L(\mathfrak{sl}_2)$.
\medskip
 
 \noindent An {\it antiautomorphism}  of a Lie algebra $\mathfrak{g}$ is a vector space  isomorphism $\alpha :\mathfrak{g} \to \mathfrak{g}$ that sends
$\lbrack u, v\rbrack \mapsto \lbrack \alpha(v), \alpha(u)\rbrack$
 for all $u,v \in \mathfrak{g}$.
We now describe an antiautomorphism $\dagger$ of $L(\mathfrak{gl}_2)$. This map  sends $u\otimes \varphi \mapsto  u^{\rm t} \otimes \tau(\varphi )$ $({\rm t}= {\rm transpose})$
for all $u \in \mathfrak{gl}_2$ and $\varphi \in \mathbb F\lbrack t, t^{-1} \rbrack$.  For $x \in L(\mathfrak{gl}_2)$ the elements $x$, $\dagger(x)$ are related as follows:
\begin{align}
x = \begin{pmatrix} a(t) & b(t) \\ c(t) & d(t) \end{pmatrix}, \qquad \qquad
\dagger(x) = \begin{pmatrix} a(t^{-1}) & c(t^{-1})\\ b(t^{-1})& d(t^{-1}) \end{pmatrix}.
\label{eq:xdagger}
\end{align} 
\noindent The automorphism $\dagger$ sends
\begin{align*}
&I \otimes t^n \mapsto I \otimes t^{-n}, \qquad \qquad
e\otimes t^n \mapsto f \otimes t^{-n}, \\
&f\otimes t^n \mapsto e \otimes t^{-n}, \qquad \qquad
h \otimes t^n \mapsto h \otimes t^{-n}
\end{align*}
for $n \in \mathbb Z$. We have $\dagger^2 = {\rm id}$.
Note that $\dagger$ leaves $L(\mathfrak{sl}_2)$ invariant, and the restriction of $\dagger$ to $L(\mathfrak{sl}_2)$ is an antiautomorphism
of $L(\mathfrak{sl}_2)$.
\medskip

\noindent In the previous paragraphs, we define the automorphisms $\theta, \sigma$ of $L(\mathfrak{gl}_2)$ and the antiautomorphism $\dagger$ of $L(\mathfrak{gl}_2)$.
The maps $\theta$, $\sigma$, $\dagger$ mutually commute.
\medskip

 \section{The Onsager Lie algebra} 
 In this section we recall the Onsager Lie algebra and review its basic properties. This material is intended to motivate Sections 4, 5
 which contain our main results. For more background information, see  \cite{chaar}.

 \begin{definition} \label{def:onsager}
 \rm (See \cite{roan}, \cite[Section~2]{dRC}.)
 Let $O$ denote the Lie subalgebra of $L(\mathfrak{sl}_2)$ consisting of the elements that are fixed by $\theta$.
 We call $O$ the {\it Onsager Lie algebra}.
 \end{definition}
 
 \begin{lemma} \label{lem:aa} The Lie algebra $O$ is invariant under $\sigma$ and $\dagger$.
 \end{lemma}
 \begin{proof} Since  $\sigma$ and $\dagger$ commute with $\theta$.
 \end{proof}
 
\noindent By Lemma \ref{lem:aa} and the construction,
 the restriction of $\sigma$ (resp. $\dagger$)  to $O$ is an automorphism (resp. antiautomorphism) of the Lie algebra $O$.
 \medskip
 
 \noindent Our next goal is to display a basis for $O$.
 
 \begin{definition}\label{def:matrices}\rm For $k \in \mathbb Z$ define $A_k, B_k \in L(\mathfrak{sl}_2)$ as follows:
 \begin{align} \label{eq:mats}
A_k =\begin{pmatrix} 0 & t^k\\ t^{-k} & 0 \end{pmatrix}, \qquad \quad
B_k = \frac{t^k - t^{-k}}{2} \begin{pmatrix} 1 & 0\\ 0 & -1 \end{pmatrix}.
\end{align}
 \end{definition}
\noindent Referring to Definition \ref{def:matrices}, we have $B_0=0$ and $B_{-k} +B_{k} =0$ for $k \in \mathbb Z$.

\begin{lemma} 
\label{def:ons} {\rm (See \cite[Section~2]{dRC}.)}
%%\cite{Onsager}
The elements
 $\lbrace A_k\rbrace_{k \in \mathbb Z}$, $\lbrace B_{k} \rbrace_{k\in \mathbb N^+}$ form a basis for the vector space $O$. Moreover for $k,\ell \in \mathbb Z$,
\begin{align}
\lbrack A_k, A_\ell \rbrack &= 2B_{k-\ell},       
 \label{eq:ComAA}
\\
\lbrack B_k, A_{\ell} \rbrack &= A_{k+\ell}- A_{\ell-k}, 
\label{eq:ComAG}
\\
\lbrack B_k, B_\ell \rbrack & = 0.
\label{eq:ComGG}
\end{align}
\end{lemma}
\begin{proof} The first assertion is readily checked using 
\eqref{eq:xtheta} and Definition \ref{def:onsager}. The relations  \eqref{eq:ComAA}--
\eqref{eq:ComGG} are verified using Definition \ref{def:matrices} and matrix multiplication.
\end{proof}

\begin{note}\rm
The basis for $O$ from Lemma \ref{def:ons}  
 is a bit nonstandard. 
 The elements $A_k, B_k$ in this basis
correspond to the elements called $A_k/2, G_k/2$
in \cite[Section~2]{dRC}. %%\cite{DateRoan2}.
We make this adjustment for notational convenience.
\end{note}
 
\begin{lemma}\label{lem:aa1} For $k \in \mathbb Z$
the automorphism $\sigma$ sends
\begin{align*}
A_k \mapsto A_{1-k}, \qquad \qquad B_k \mapsto B_{-k},
\end{align*}
and the antiautomorphism $\dagger$  sends
\begin{align*}
A_k \mapsto A_k, \qquad \qquad B_k \mapsto B_{-k}.
\end{align*}
\end{lemma}
\begin{proof} The first assertion is verified using Definition \ref{def:matrices} and \eqref{eq:xsigma}. The second assertion is verified using Definition \ref{def:matrices} and \eqref{eq:xdagger}.
\end{proof}

\noindent We give another basis for $O$.

\begin{definition} \label{def:WWGG}
\rm For $k \in \mathbb N$ define
\begin{align}
\label{eq:WW}
&W_{-k} = \biggl(\frac{t+t^{-1}}{2} \biggr)^k \begin{pmatrix} 0 &1 \\ 1 & 0 \end{pmatrix}, \qquad \quad 
W_{k+1} = \biggl(\frac{t+t^{-1}}{2} \biggr)^k \begin{pmatrix} 0 &t \\ t^{-1} & 0 \end{pmatrix}, \\
&\tilde G_{k+1} = \biggl(\frac{t+t^{-1}}{2} \biggr)^k \begin{pmatrix} t^{-1}-t&0 \\ 0 & t-t^{-1} \end{pmatrix}.
\label{eq:tG}
\end{align}
Note that $W_0=A_0$ and $W_1=A_1$.
\end{definition}

\noindent The following basis is a variation on \cite[Definition~4.1 and Theorem~2]{BC17}.

\begin{lemma} \label{lem:OA2} The elements
\begin{align}
\label{eq:basisWWG}
\lbrace W_{-k}\rbrace_{k \in \mathbb N}, \quad \lbrace W_{k+1}\rbrace_{k \in \mathbb N}, \quad \lbrace {\tilde G}_{k+1}\rbrace_{k \in \mathbb N}
\end{align}
form a basis for the vector space $O$. Moreover for $k, \ell \in \mathbb N$,
\begin{align}
\lbrack W_{-k}, W_{\ell+1} \rbrack &= \tilde G_{k+\ell+1}, \label{eq:com1}
\\
\lbrack \tilde G_{k+1}, W_{-\ell} \rbrack &= 4 W_{-k-\ell-1} - 4 W_{k+\ell+1}, \label{eq:com2}
\\
\lbrack W_{k+1}, \tilde G_{\ell+1} \rbrack &= 4 W_{k+\ell+2} - 4W_{-k-\ell}, \label{eq:com3}
\\
\lbrack W_{-k}, W_{-\ell} \rbrack &=0, \qquad 
\lbrack W_{k+1}, W_{\ell+1} \rbrack=0, \qquad
 \lbrack \tilde G_{k+1}, \tilde G_{\ell+1} \rbrack=0.
 \label{eq:com456}
 \end{align}
 \end{lemma}
\begin{proof} The first assertion is readily checked using 
\eqref{eq:xtheta} and Definition \ref{def:onsager}. The relations  \eqref{eq:com1}--
\eqref{eq:com456} are verified using Definition \ref{def:WWGG} and matrix multiplication.
\end{proof}

\begin{note}\rm  The notations of Lemma \ref{lem:OA2} and \cite[Definition~4.1]{BC17} are related as follows.
The elements $W_{-k}$, $W_{k+1}$, $\tilde G_{k+1}$ in Lemma \ref{lem:OA2} correspond to the elements called 
$W_{-k}/2$, $W_{k+1}/2$, $\tilde G_{k+1}/4$ in
 \cite[Definition~4.1]{BC17}.
 \end{note}

\begin{lemma} \label{lem:aaWG}
For $k \in \mathbb N$, the automorphism $\sigma$ sends
\begin{align*}
W_{-k} \mapsto W_{k+1}, 
\qquad \qquad 
W_{k+1} \mapsto W_{-k}, \qquad \qquad 
{\tilde G}_{k+1} \mapsto -{\tilde G}_{k+1},
\end{align*}
and the antiautomorphism $\dagger$ sends
\begin{align*}
W_{-k} \mapsto W_{-k}, 
\qquad \qquad 
W_{k+1} \mapsto W_{k+1}, \qquad \qquad 
{\tilde G}_{k+1} \mapsto -{\tilde G}_{k+1}.
\end{align*}
\end{lemma}
\begin{proof} The first assertion is verified using Definition \ref{def:WWGG} and \eqref{eq:xsigma}. The second assertion is verified using Definition \ref{def:WWGG} and \eqref{eq:xdagger}.
\end{proof}

\noindent We have a comment.
\begin{lemma}\label{lem:centerO}   %%{\rm (See CITE.)}
The center of $O$ is equal to $0$.
\end{lemma}
\begin{proof} A matrix $x$ in the center of $O$ has trace 0, and commutes with $A_0$ and $A_1$. By matrix multiplication we find $x=0$.
\end{proof}

 \noindent  We mention two presentations of $O$ by generators and relations.

 \begin{lemma}
 \label{def:OA}  {\rm (See \cite[Section~2]{dRC}.)}
 The Lie algebra $O$ has a presentation by generators $W_0, W_1$ and relations
 \begin{align} \label{eq:DG}
\lbrack W_0, \lbrack W_0, \lbrack W_0, W_1\rbrack \rbrack \rbrack =4\lbrack W_0, W_1 \rbrack,
\qquad \quad
\lbrack W_1, \lbrack W_1, \lbrack W_1, W_0\rbrack \rbrack\rbrack = 4\lbrack W_1 W_0 \rbrack.
\end{align}
 \end{lemma}
 \begin{note}\rm 
  The relations  \eqref{eq:DG} are called the {\it Dolan/Grady relations} \cite{Dolgra}.
  \end{note}

%%%%\noindent The following result is a consequence of  \cite[Definition~4.1 and Theorem~2]{BC17}; we give a short proof for the sake of completeness.
\begin{lemma} \label{lem:Olong} The Lie algebra $O$
has a presentation 
by generators
\begin{align}
\label{eq:4gensO}
\lbrace W_{-k}\rbrace_{k \in \mathbb N}, \quad 
\lbrace W_{k+1}\rbrace_{k \in \mathbb N}, \quad
 \lbrace {\tilde G}_{k+1}\rbrace_{k \in \mathbb N}
\end{align}
 and the following relations. For $k, \ell \in \mathbb N$,
\begin{align*}
&
 \lbrack W_0, W_{k+1}\rbrack= 
\lbrack W_{-k}, W_{1}\rbrack=
{\tilde G}_{k+1},
%%\label{eq:3p1}
\\
&
\lbrack {\tilde G}_{k+1},  W_{0}\rbrack= 
4 W_{-k-1}-4
 W_{k+1},
%%\label{eq:3p2}
\\
&
\lbrack  W_{1}, {\tilde G}_{k+1}\rbrack= 
4W_{k+2}-4 
 W_{-k},
%\label{eq:3p3}
\\
&
\lbrack W_{-k}, W_{-\ell}\rbrack=0,  \qquad 
\lbrack W_{k+1}, W_{\ell+1}\rbrack= 0,
%%\label{eq:3p4}
\\
&
\lbrack W_{-k}, W_{\ell+1}\rbrack+
\lbrack W_{k+1}, W_{-\ell}\rbrack= 0,
%%\label{eq:3p5}
\\
&
\lbrack W_{-k},  {\tilde G}_{\ell+1}\rbrack+
\lbrack {\tilde G}_{k+1},  W_{-\ell}\rbrack= 0,
%%\label{eq:3p7}
\\
&
\lbrack W_{k+1}, {\tilde G}_{\ell+1}\rbrack+
\lbrack {\tilde G}_{k+1}, W_{\ell+1}\rbrack= 0,
%\label{eq:3p9}
\\
&
\lbrack {\tilde G}_{k+1},  {\tilde G}_{\ell+1}\rbrack= 0.
\end{align*}
\end{lemma}
\begin{proof} Define a Lie algebra $O^\vee$ by the presentation in the lemma statement. We show that $O^\vee$ is isomorphic to $O$.
Setting $(k,\ell)=(0,1)$ in the first four relations of the lemma statement, we routinely find that the elements $W_0$, $W_1$ of $O^\vee$ satisfy the Dolan/Grady relations \eqref{eq:DG}. 
By this and Lemma  \ref{def:OA},
 there exists a Lie algebra homomorphism $\alpha: O \to O^\vee$
that sends $W_0 \mapsto W_0$ and $W_1 \mapsto W_1$. The homomorphism $\alpha$ is surjective, because the $O^\vee$-generators \eqref{eq:4gensO} are contained in $\alpha (O)$  by the first three relations in the lemma statement and induction on the ordering
$W_0, W_1, \tilde G_1, W_{-1}, W_2, \tilde G_2, \ldots $
 By matrix multiplication, the elements of $O$  from Definition \ref{def:WWGG} satisfy the relations in the lemma statement. Therefore, there exists
a Lie algebra homomorphism $\beta: O^\vee \to O$ that sends 
\begin{align*}
W_{-k} \mapsto W_{-k}, \qquad \quad
W_{k+1} \mapsto W_{k+1}, \qquad \quad
\tilde G_{k+1} \mapsto {\tilde G}_{k+1}
\end{align*}
for $k \in \mathbb N$. By construction $\beta$ is surjective. The homomorphisms $\alpha$, $\beta$ are inverses, and hence bijections. The result follows.
\end{proof}

\begin{note}\rm There is another proof of Lemma \ref{lem:Olong} that uses Lemma \ref{lem:OA2}. The proof strategy  is to show that the following imply each other: (i) the relations \eqref{eq:com1}--\eqref{eq:com456};
 (ii) the relations listed in Lemma \ref{lem:Olong}. These implications can be obtained by adapting the proof of \cite[Proposition~2.18]{FMA}.
\end{note}

\section{The alternating central extension $\mathcal O$ of $O$}

\noindent In this section we introduce the alternating central extension $\mathcal O$ of $O$.
\medskip

\noindent  To motivate things, we have some comments about the loop algebras $L(\mathfrak{sl}_2)$ and $L(\mathfrak{gl}_2)$. Let $C$ denote the center of $L(\mathfrak{gl}_2)$. We have $C=I \otimes \mathbb F\lbrack t, t^{-1} \rbrack$.
  The sum $L(\mathfrak{gl}_2)=
  C
   + L(\mathfrak{sl}_2)$ is direct, and both summands are ideals in $L(\mathfrak{gl}_2)$.
 % The quotient Lie algebra $L(\mathfrak{gl}_2)/L(\mathfrak{sl}_2)$ is isomorphic to 
 %  $C$,
 % and hence abelian. The quotient Lie algebra  $L(\mathfrak{gl}_2)/C
%  $ is isomorphic to $L(\mathfrak{sl}_2)$.
\medskip

\noindent  There exists a surjective algebra homomorphism
$\gamma: \mathbb F \lbrack t,t^{-1} \rbrack \to \mathbb F$ that sends $t\mapsto 1$. For $\varphi  \in \mathbb F \lbrack t,t^{-1} \rbrack$ the scalar $\gamma(\varphi)$ is described as follows.
Write $\varphi = \sum_{n \in \mathbb Z} a_n t^n$ with $a_n \in \mathbb F$ for $n \in \mathbb Z$. Then $\gamma(\varphi) = \sum_{n \in \mathbb Z} a_n$. Let 
$\mathbb F \lbrack t,t^{-1} \rbrack'$ denote the kernel of $\gamma$. The sum
$\mathbb F \lbrack t,t^{-1} \rbrack=
\mathbb F 1 + \mathbb F \lbrack t,t^{-1} \rbrack'$ is direct. The vector space $\mathbb F \lbrack t,t^{-1} \rbrack'$ has a basis $\lbrace t^n-1 \rbrace_{0 \not=n \in \mathbb Z}$.
Define $C'=I \otimes \mathbb F \lbrack t,t^{-1} \rbrack'$. The sum $C=
\mathbb F I\otimes 1+ C'$ is direct. 
The vector space $C'$ has basis $\lbrace I \otimes (t^n-1) \rbrace_{0 \not=n \in \mathbb Z}$. 
\medskip

\noindent
There exists a surjective Lie algebra homomorphism $\varepsilon : L(\mathfrak{gl}_2) \to \mathbb F$ that sends $u \otimes \varphi \mapsto {\rm tr}(u) \gamma(\varphi)$ for all $u \in \mathfrak{gl}_2 $ and $\varphi \in \mathbb F\lbrack t, t^{-1} \rbrack$.
 For $x \in L(\mathfrak{gl}_2)$ the scalar
$\varepsilon(x)$ is described as follows:
\begin{align} \label{eq:xe}
x = \begin{pmatrix} a(t) & b(t) \\ c(t) & d(t) \end{pmatrix}, \qquad \quad \varepsilon(x) = a(1)+d(1).
\end{align}
\noindent The homomorphism $\varepsilon$ sends
\begin{align*}
&I \otimes t^n \mapsto 2, \qquad \qquad
e\otimes t^n \mapsto 0, \qquad \qquad 
f\otimes t^n \mapsto 0, \qquad \qquad
h \otimes t^n \mapsto 0
\end{align*}
for $n \in \mathbb Z$. Comparing \eqref{eq:xtheta}, \eqref{eq:xsigma}, \eqref{eq:xdagger} with \eqref{eq:xe}, we find that the following diagrams commute:
\begin{align*}
{\begin{CD}
L(\mathfrak{gl}_2) @>\varepsilon  >> \mathbb F
              \\
         @V \theta VV                   @VV {\rm id}V \\
 L(\mathfrak{gl}_2) @>>\varepsilon >
                                 \mathbb F
                        \end{CD}}  	
         \qquad \qquad                
    {\begin{CD}
L(\mathfrak{gl}_2)  @>\varepsilon  >> \mathbb F
              \\
         @V \sigma VV                   @VV {\rm id} V \\
 L(\mathfrak{gl}_2) @>>\varepsilon >
                                 \mathbb F
                        \end{CD}}  	     \qquad \qquad                
   {\begin{CD}
L(\mathfrak{gl}_2)  @>\varepsilon  >> \mathbb F
              \\
         @V \dagger VV                   @VV {\rm id} V \\
 L(\mathfrak{gl}_2) @>>\varepsilon >
                                 \mathbb F
                        \end{CD}}  	                                   		    
\end{align*}
%%\noindent Here is another point of view.
%%Write 
%%$x = \sum_{n \in \mathbb Z} x_n \otimes t^n$ with $x_n \in \mathfrak{gl}_2$ for $n \in \mathbb Z$.
%%Then $\varepsilon(x)= \sum_{n \in \mathbb Z} {\rm tr}(x_n)$.
\noindent  Let $\mathcal L$ denote the kernel of $\varepsilon$.  The sum $L(\mathfrak{gl}_2) = \mathbb F I\otimes 1 + \mathcal L$ is direct, and both summands are ideals of $L(\mathfrak{gl}_2)$.
The  sum $\mathcal L = C' + L(\mathfrak{sl}_2)$ is direct, and both summands are ideals of
$\mathcal L$.
%The quotient Lie algebra 
%$\mathcal L/L(\mathfrak{sl}_2)$ is isomorphic to $C'$, and hence abelian.
%The quotient Lie algebra  $\mathcal L/C'$ is isomorphic to $L(\mathfrak{sl}_2)$.
By the above commuting diagrams, the Lie algebra $\mathcal L$ is invariant under $\theta$, $\sigma$, $\dagger$.

\begin{definition}\label{def:OO} \rm Let $\mathcal O$ denote the Lie algebra consisting of the elements in $\mathcal L$ that are fixed by $\theta$. 
As we will see, there exists a surjective Lie algebra homomorphism $\mathcal O \to O$ whose kernel is the center of $\mathcal O$.
%%%By construction, $O$ is an ideal of $\mathcal O$ and the quotient Lie algebra $\mathcal O/O$ is abelian.
We call $\mathcal O$ the {\it alternating central extension of $O$}.
\end{definition}

 \begin{lemma} \label{lem:aaa} The Lie algebra $\mathcal O$ is invariant under $\sigma$ and $\dagger$.
 \end{lemma}
 \begin{proof} Since  $\sigma$ and $\dagger$ commute with $\theta$.
 \end{proof}
 
\noindent By Lemma \ref{lem:aaa} and the construction,
 the restriction of $\sigma$ (resp. $\dagger$)  to $\mathcal O$ is an automorphism (resp. antiautomorphism) of the Lie algebra $\mathcal O$.
 \medskip

\noindent Our next goal is to display a basis for $\mathcal O$.

 \begin{definition}\label{def:matricesACE} For $k \in \mathbb Z$ define
 \begin{align} \label{eq:matsACE}
\mathcal A_k =\begin{pmatrix} 0 & t^k\\ t^{-k} & 0 \end{pmatrix}, \qquad \quad
\mathcal B_k =  \begin{pmatrix} t^k-1 & 0\\ 0 & t^{-k}-1 \end{pmatrix}.
\end{align}
 \end{definition}
\noindent Referring to Definition \ref{def:matricesACE}, we have $\mathcal B_0=0$.

\begin{lemma} 
\label{def:onsACE} 
%%\cite{Onsager}
The elements
 $\lbrace \mathcal A_k\rbrace_{k \in \mathbb Z}$, $\lbrace \mathcal B_{k} \rbrace_{0 \not=k \in \mathbb Z}$ form a basis for $\mathcal O$. Moreover for $k,\ell \in \mathbb Z$,
\begin{align}
\lbrack \mathcal A_k, \mathcal A_\ell \rbrack &= \mathcal B_{k-\ell}-\mathcal B_{\ell-k},      
 \label{eq:ComAAACE}
\\
\lbrack \mathcal B_k, \mathcal A_{\ell} \rbrack &= A_{k+\ell}- A_{\ell-k}, 
\label{eq:ComAGACE}
\\
\lbrack \mathcal B_k, \mathcal B_\ell \rbrack & = 0.
\label{eq:ComGGACE}
\end{align}
\end{lemma} 
\begin{proof}The first assertion is readily checked using \eqref{eq:xtheta}, \eqref{eq:xe} and Definition \ref{def:OO}. The relations
 \eqref{eq:ComAAACE}--\eqref{eq:ComGGACE} are verified using Definition \ref{def:matricesACE} and matrix multiplication.
 \end{proof}

\noindent Next we clarify how the elements $\lbrace A_k\rbrace_{k \in \mathbb Z}$, $\lbrace B_{k} \rbrace_{k\in \mathbb Z}$ from Definition \ref{def:matrices}
are related to the elements  $\lbrace \mathcal A_k\rbrace_{k \in \mathbb Z}$, $\lbrace \mathcal B_{k} \rbrace_{k \in \mathbb Z}$ from Definition \ref{def:matricesACE}.
\begin{lemma} \label{lem:AABB} For $k \in \mathbb Z$,
\begin{align*}
A_k = \mathcal A_k, \qquad \qquad B_k = (\mathcal B_k - \mathcal B_{-k})/2.
\end{align*}
\end{lemma}
\begin{proof} Use \eqref{eq:mats} and  \eqref{eq:matsACE}.
\end{proof}

\begin{lemma}\label{lem:aaa1} For $k \in \mathbb Z$
the automorphism $\sigma$ sends
\begin{align*}
\mathcal A_k \mapsto \mathcal A_{1-k}, \qquad \qquad \mathcal B_k \mapsto  \mathcal B_{-k},
\end{align*}
and the antiautomorphism $\dagger$  sends
\begin{align*}
\mathcal A_k \mapsto \mathcal A_k, \qquad \qquad \mathcal B_k \mapsto \mathcal B_{-k}.
\end{align*}
\end{lemma}
\begin{proof} The first assertion is verified using Definition \ref{def:matricesACE} and \eqref{eq:xsigma}. The second assertion is verified using Definition \ref{def:matricesACE} and \eqref{eq:xdagger}.
\end{proof}

\noindent Next we consider the center of the Lie algebra $\mathcal O$.
\begin{definition}\label{def:ZZ}
\rm Let $\mathcal Z$ denote the center of $\mathcal O$.
\end{definition}
\begin{lemma} \label{def:ZZ1} For  $k\in \mathbb Z$ we have
\begin{align*}
\mathcal B_k+ \mathcal B_{-k} = I \otimes (t^k + t^{-k}-2).
\end{align*}
Moreover, $\mathcal B_k+ \mathcal B_{-k} \in \mathcal Z$.
\end{lemma}
\begin{proof} The first assertion is verified using \eqref{eq:matsACE}. The second assertion is clear.
\end{proof}

\begin{lemma}  \label{lem:ZZZ} The following {\rm (i)--(iv)} hold.
\begin{enumerate}
\item[\rm (i)] The vector space $O$ has a basis
\begin{align}\label{eq:Obasis2}
&\mathcal A_k, \quad k \in \mathbb Z;
\qquad \qquad 
(\mathcal B_k - \mathcal B_{-k})/2, \quad k\in \mathbb N^+.
\end{align}
\item[\rm (ii)] The vector space $\mathcal Z$ has a basis
\begin{align} \label{eq:Zbasis}
(\mathcal B_k + \mathcal B_{-k})/2, \qquad  k\in \mathbb N^+.
\end{align}
\item[\rm (iii)] The sum $\mathcal O= O + \mathcal Z$ is direct.
\item[\rm (iv)] $O$ is an ideal of $\mathcal O$.
\end{enumerate}
\end{lemma} 
\begin{proof} (i) By Lemmas \ref{def:ons}, \ref{lem:AABB}.
\\ (ii), (iii) Let $\mathcal Z^\vee$ denote the subspace of $\mathcal O$ spanned by \eqref{eq:Zbasis}. The sum $\mathcal O= O + \mathcal Z^\vee$ is direct, by the first assertion of Lemma \ref{def:onsACE} together with 
\eqref{eq:Obasis2}, \eqref{eq:Zbasis}.  We have $\mathcal Z \cap O=0$ by Lemma  \ref{lem:centerO}, and $\mathcal Z^\vee \subseteq \mathcal Z$ by Lemma
\ref{def:ZZ1}. By these comments and linear algebra, we obtain $\mathcal Z^\vee=\mathcal Z$. The results follow.
\\ (iv) By (iii) above, together with the fact that $O$ is a Lie subalgebra of $\mathcal O$ and  $\lbrack O, \mathcal Z\rbrack=0$.
\end{proof}

\begin{lemma} \label{lem:centerFix}
The automorphism $\sigma$ and the antiautomorphism $\dagger$ fix everything in $\mathcal Z$.
\end{lemma} 
\begin{proof} By Lemma \ref{lem:aaa1} and Lemma \ref{lem:ZZZ}(ii).
\end{proof}

\begin{lemma} \label{lem:Ogen} The Lie algebra $\mathcal O$ is generated by $\mathcal W_0$, $\mathcal W_1$, $\mathcal Z$.
\end{lemma}
\begin{proof} By Lemma \ref{def:OA} and Lemma  \ref{lem:ZZZ}(iii).
\end{proof}

\noindent In Lemma \ref{def:onsACE} we gave a basis for $\mathcal O$. Next we display another basis for $\mathcal O$.

\begin{definition}\label{def:WWGGACE}
\rm For $k \in \mathbb N$ define
\begin{align}
\label{eq:WWACE}
&\mathcal W_{-k} = \biggl(\frac{t+t^{-1}}{2} \biggr)^k \begin{pmatrix} 0 &1 \\ 1 & 0 \end{pmatrix}, \qquad \quad 
\mathcal W_{k+1} = \biggl(\frac{t+t^{-1}}{2} \biggr)^k \begin{pmatrix} 0 &t \\ t^{-1} & 0 \end{pmatrix}, \\
&
\mathcal G_{k+1} = \biggl(\frac{t+t^{-1}}{2} \biggr)^k \begin{pmatrix} 2t-2&0 \\ 0 & 2t^{-1}-2 \end{pmatrix},
\label{eq:GACE}
\\
&
\mathcal {\tilde G}_{k+1} = \biggl(\frac{t+t^{-1}}{2} \biggr)^k \begin{pmatrix} 2t^{-1}-2&0 \\ 0 & 2t-2 \end{pmatrix}.
\label{eq:tGACE}
\end{align}
\end{definition}

\begin{lemma} \label{lem:OA3} The elements 
\begin{align}\label{eq:basisWWGG}
\lbrace \mathcal W_{-k}\rbrace_{k \in \mathbb N}, \quad 
\lbrace \mathcal W_{k+1}\rbrace_{k \in \mathbb N}, \quad
\lbrace \mathcal G_{k+1}\rbrace_{k \in \mathbb N}, \quad
 \lbrace \mathcal {\tilde G}_{k+1}\rbrace_{k \in \mathbb N}
 \end{align}
form a basis for the vector space $\mathcal O$. Moreover for $k, \ell \in \mathbb N$,
\begin{align}
&\lbrack \mathcal W_{-k}, \mathcal W_{\ell+1} \rbrack =(\mathcal {\tilde G}_{k+\ell+1}-\mathcal G_{k+\ell+1})/2, \label{eq:ACEcom1}
\\
&\lbrack \mathcal W_{-\ell}, \mathcal G_{k+1} \rbrack = 
\lbrack \mathcal {\tilde G}_{k+1}, \mathcal W_{-\ell} \rbrack = 4 \mathcal W_{-k-\ell-1} - 4 \mathcal W_{k+\ell+1}, \label{eq:ACEcom2}
\\
&\lbrack \mathcal G_{k+1}, \mathcal W_{\ell+1} \rbrack = 
\lbrack \mathcal W_{\ell+1}, \mathcal {\tilde G}_{k+1} \rbrack = 4 \mathcal W_{k+\ell+2} - 4 \mathcal W_{-k-\ell}, \label{eq:ACEcom3}
\\
&\lbrack \mathcal W_{-k}, \mathcal W_{-\ell} \rbrack =0, \qquad 
\lbrack \mathcal W_{k+1}, \mathcal W_{\ell+1} \rbrack=0, \qquad
 \lbrack \mathcal {G}_{k+1}, \mathcal {G}_{\ell+1} \rbrack=0,
 \label{eq:ACEcom456}
 \\
 &
 \lbrack \mathcal {G}_{k+1}, \mathcal {\tilde G}_{\ell+1} \rbrack=0, \qquad 
 \lbrack \mathcal {\tilde G}_{k+1}, \mathcal {\tilde G}_{\ell+1} \rbrack=0.
 \label{eq:ACEcom78}
 \end{align}
 \end{lemma}
 \begin{proof}
 The first assertion is readily checked using \eqref{eq:xtheta}, \eqref{eq:xe} and Definition \ref{def:OO}. The relations
  \eqref{eq:ACEcom1}--\eqref{eq:ACEcom78}
 are verified using Definition \ref{def:WWGGACE} and matrix multiplication.
 \end{proof}

\noindent Next we clarify how the elements $\lbrace W_{-k}\rbrace_{k \in \mathbb N}$,
$\lbrace W_{k+1}\rbrace_{k \in \mathbb N}$,
$\lbrace {\tilde G}_{k+1} \rbrace_{k\in \mathbb N}$ from Definition \ref{def:WWGG}
are related to the elements
 $\lbrace \mathcal W_{-k}\rbrace_{k \in \mathbb N}$,
$\lbrace \mathcal W_{k+1}\rbrace_{k \in \mathbb N}$,
$\lbrace \mathcal G_{k+1} \rbrace_{k\in \mathbb N}$,
$\lbrace \mathcal {\tilde G}_{k+1} \rbrace_{k\in \mathbb N}$
from Definition \ref{def:WWGGACE}.

\begin{lemma} \label{lem:WWGGcom} For $k \in \mathbb N$,
\begin{align*}
W_{-k} = \mathcal W_{-k}, \qquad \qquad 
W_{k+1} = \mathcal W_{k+1}, \qquad \qquad 
{\tilde G}_{k+1} = (\mathcal {\tilde G}_{k+1} - \mathcal G_{k+1})/2.
\end{align*}
\end{lemma}
\begin{proof} Compare Definition \ref{def:WWGG} and Definition \ref{def:WWGGACE}.
\end{proof}

\begin{lemma} \label{lem:aaWG1}
For $k \in \mathbb N$, the automorphism $\sigma$ sends
\begin{align*}
\mathcal W_{-k} \mapsto \mathcal W_{k+1}, 
\qquad \quad 
\mathcal W_{k+1} \mapsto \mathcal W_{-k}, \quad \qquad 
\mathcal {\tilde G}_{k+1} \mapsto \mathcal {G}_{k+1}, \quad \qquad 
\mathcal {G}_{k+1} \mapsto \mathcal {\tilde G}_{k+1}, 
\end{align*}
and the antiautomorphism $\dagger$ sends
\begin{align*}
\mathcal W_{-k} \mapsto \mathcal W_{-k}, 
\qquad \quad 
\mathcal W_{k+1} \mapsto \mathcal W_{k+1}, \quad \qquad 
\mathcal {\tilde G}_{k+1} \mapsto \mathcal {G}_{k+1}, \qquad \quad
\mathcal {G}_{k+1} \mapsto \mathcal {\tilde G}_{k+1}.
\end{align*}
\end{lemma}
\begin{proof} The first assertion is verified using Definition \ref{def:WWGGACE} and \eqref{eq:xsigma}. The second assertion is verified using  Definition \ref{def:WWGGACE} and \eqref{eq:xdagger}.
\end{proof}

\noindent We return our attention to the center $\mathcal Z$ of $\mathcal O$.

\begin{lemma} For  $k\in \mathbb N$ we have
\begin{align*}
(\mathcal G_{k+1}+ \mathcal {\tilde G}_{k+1})/2 = I \otimes (t+t^{-1}-2)  \biggl(\frac{t+t^{-1}}{2} \biggr)^k.
\end{align*}
Moreover, $(\mathcal G_{k+1}+ \mathcal {\tilde G}_{k+1})/2 \in \mathcal Z$.
\end{lemma}
\begin{proof} The first assertion is verified using \eqref{eq:GACE}, \eqref{eq:tGACE}.
The second assertion is clear.
\end{proof}

\begin{lemma}  \label{lem:ZZZACE} The following {\rm (i), (ii)} hold.
\begin{enumerate}
\item[\rm (i)] The vector space  $O$ has a basis
\begin{align}
&\mathcal W_{-k}, \qquad \mathcal W_{k+1}, 
\qquad 
(\mathcal {\tilde G}_{k+1} - \mathcal G_{k+1})/2, \qquad\qquad k\in \mathbb N.
\label{eq:GG2}
\end{align}
\item[\rm (ii)] The vector space $\mathcal Z$ has a basis
\begin{align}
(\mathcal {\tilde G}_{k+1} + \mathcal G_{k+1})/2, \qquad  k\in \mathbb N.
\label{eq:ZGG2}
\end{align}
\end{enumerate}
\end{lemma}
\begin{proof}
Similar to the proof of Lemma  \ref{lem:ZZZ}.
\end{proof} 

\noindent In the next two results, we clarify how $O$ and $\mathcal O$ are related.

\begin{lemma} \label{lem:inj} The following {\rm (i)--(iv)} hold:
\begin{enumerate}
\item[\rm (i)] 
there exists a unique Lie algebra homomorphism $\iota: O\to \mathcal O$ that sends $W_0 \mapsto \mathcal W_0$ and $W_1 \mapsto \mathcal W_1$;
\item[\rm (ii)] the map $\iota$ is injective;
\item[\rm (iii)] the map $\iota$ sends
\begin{align*}
A_k \mapsto \mathcal A_k, \qquad \qquad B_k \mapsto (\mathcal B_k - \mathcal B_{-k})/2
\end{align*}
for $k \in \mathbb Z$.
\item[\rm (iv)] the map $\iota$ sends
\begin{align*}
W_{-k} \mapsto \mathcal W_{-k}, \qquad
W_{k+1} \mapsto \mathcal W_{k+1}, \qquad
\tilde G_{k+1} \mapsto (\mathcal {\tilde G}_{k+1} - \mathcal G_{k+1})/2
\end{align*}
\noindent for $k \in \mathbb N$.
\end{enumerate}
\end{lemma} 
\begin{proof}  The inclusion map $\iota: O\to \mathcal O$ is an injective Lie algebra homomorphism. The map $\iota$ satisfies (iii) by Lemma  \ref{lem:AABB}.
 The map $\iota$ satisfies (iv) by Lemma \ref{lem:WWGGcom}. The map $\iota$ sends
$W_0 \mapsto \mathcal W_0$ and $W_1 \mapsto \mathcal W_1$. The uniqueness assertion in (i) is satisfied because $O$ is generated by $W_0$, $W_1$ in view of Lemma  \ref{def:OA}.
\end{proof}

\begin{lemma}\label{lem:surj} The following {\rm (i)--(v)} hold:
\begin{enumerate}
\item[\rm (i)] there exists a unique  Lie algebra homomorphism $\rho: \mathcal O \to O$ that sends $\mathcal W_0 \mapsto W_0$ and $\mathcal W_1 \mapsto W_1$ and $\mathcal Z \mapsto 0$;
\item[\rm (ii)] the map $\rho$ is surjective;
\item[\rm (iii)] the map $\rho$ has kernel $\mathcal Z$;
\item[\rm (iv)] the map $\rho$ sends
$\mathcal A_k \mapsto A_k$ and $ \mathcal B_k \mapsto B_k$
for $k \in \mathbb Z$;
\item[\rm (v)] the map $\rho$ sends
\begin{align*}
\mathcal W_{-k} \mapsto W_{-k}, \qquad
\mathcal W_{k+1} \mapsto W_{k+1}, \qquad
\mathcal {\tilde G}_{k+1} \mapsto \tilde G_{k+1}, \qquad
\mathcal G_{k+1} \mapsto -\tilde G_{k+1}
\end{align*}
\noindent for $k \in \mathbb N$.
\end{enumerate}
\end{lemma}
\begin{proof} By Lemma  \ref{lem:ZZZ}(iii) the sum $\mathcal O = O+\mathcal Z$ is direct. So there exists an $\mathbb F$-linear map 
 $\rho: \mathcal O \to O$ that acts as the identity on $O$ and zero on $\mathcal Z$. The map $\rho$ is surjective and has kernel $\mathcal Z$. The map $\rho$ is a Lie
 algebra homomorphism, because $\mathcal Z$ is an ideal in $\mathcal O$.
 The map $\rho$ satisfies (iv), by  Lemma  \ref{lem:AABB} and Lemma   \ref{lem:ZZZ}.
The map $\rho$ satisfies (v), by  Lemma \ref{lem:WWGGcom} and Lemma \ref{lem:ZZZACE}. The map $\rho$ sends
$\mathcal W_0 \mapsto W_0$ and $\mathcal W_1 \mapsto W_1$ and $\mathcal Z \mapsto 0$. The uniqueness assertion in (i) is satisfied by Lemma \ref{lem:Ogen}. 
\end{proof}

\noindent In Lemma  \ref{lem:ZZZ} we saw the direct sum $\mathcal O = O+\mathcal Z$. Next, we give two more direct sum decompositions of $\mathcal O$.

\begin{lemma} \label{lem:OG} The following {\rm (i)--(iv)} hold:
\begin{enumerate}
\item[\rm (i)] 
there exists an abelian Lie subalgebra $G$ of $\mathcal O$ that has basis $\lbrace \mathcal {G}_{k+1} \rbrace_{k \in \mathbb N}$;
\item[\rm (ii)] the sum $\mathcal O = O + G$ is direct;
\item[\rm (iii)] 
there exists an abelian Lie subalgebra $\tilde G$ of $\mathcal O$ that has basis $\lbrace \mathcal {\tilde G}_{k+1} \rbrace_{k \in \mathbb N}$;
\item[\rm (iv)] the sum $\mathcal O = O + {\tilde G}$ is direct.
\end{enumerate}
\end{lemma}
\begin{proof} (i) By Lemma  \ref{lem:OA3}  the elements
 $\lbrace \mathcal {G}_{k+1} \rbrace_{k \in \mathbb N}$ are linearly independent, and 
 $\lbrack \mathcal {G}_{k+1}, \mathcal {G}_{\ell+1} \rbrack=0$ for $k, \ell \in \mathbb N$. The result follows.
 \\
 \noindent (ii) By Lemma \ref{lem:OA3} the elements 
\eqref{eq:basisWWGG} form a basis for $\mathcal O$. By this and linear algebra, the elements
\begin{align*}
\mathcal W_{-k}, \qquad
\mathcal W_{k+1}, \qquad
(\mathcal {\tilde G}_{k+1} - \mathcal G_{k+1})/2, \qquad 
\mathcal {G}_{k+1} \qquad \qquad k \in \mathbb N
\end{align*}
form a basis for $\mathcal O$. The result follows in view of Lemma  \ref{lem:ZZZACE}(i).
\\
\noindent (iii), (iv) Apply the automorphism $\sigma$ to everything in (i), (ii) above, and use Lemma \ref{lem:aaWG1}.
\end{proof}

\begin{lemma} \label{lem:moreGen}
The Lie algebra $\mathcal O$ is generated by $\mathcal W_0$, $\mathcal W_1$, $\lbrace \mathcal {G}_{k+1}\rbrace_{k\in \mathbb N}$ and also by
$\mathcal W_0$, $\mathcal W_1$, $\lbrace \mathcal {\tilde G}_{k+1}\rbrace_{k\in \mathbb N}$.
\end{lemma}
\begin{proof} By Lemma \ref{lem:OG} and since the Lie algebra $O$ is generated by $\mathcal W_0$, $\mathcal W_1$.
\end{proof}

\noindent Next, we give some presentations of $\mathcal O$ by generators and relations.

\begin{lemma} \label{lem:Compacta} The Lie algebra $\mathcal O$ has a presentation by generators $\mathcal W_0$, $\mathcal W_1$, $\lbrace \mathcal {G}_{k+1} \rbrace_{k\in \mathbb N}$ and
the following relations:
\begin{align}
&\lbrack \mathcal W_0, \lbrack \mathcal W_0, \lbrack \mathcal W_0, \mathcal W_1 \rbrack \rbrack \rbrack = 4 \lbrack \mathcal W_0, \mathcal W_1 \rbrack, \label{eq:G1}
\\
&\lbrack \mathcal W_1, \lbrack \mathcal W_1, \lbrack \mathcal W_1, \mathcal W_0 \rbrack \rbrack \rbrack = 4 \lbrack \mathcal W_1, \mathcal W_0 \rbrack, \label{eq:G2}
\\
&\lbrack \mathcal W_1, \lbrack \mathcal W_1, \mathcal W_0 \rbrack \rbrack  =  \lbrack \mathcal W_1, \mathcal {G}_1 \rbrack, \label{eq:G3}
\\
&\lbrack \lbrack \mathcal W_1, \mathcal W_0\rbrack, \mathcal W_0 \rbrack   =  \lbrack \mathcal {G}_1, \mathcal W_0 \rbrack, \label{eq:G4}
\\
&\lbrack \mathcal W_1, \lbrack \mathcal W_1, \lbrack \mathcal W_0, \mathcal {G}_k \rbrack \rbrack \rbrack = 4 \lbrack \mathcal W_1, \mathcal {G}_{k+1} \rbrack \qquad \qquad k\in \mathbb N^+, \label{eq:G5}
\\
&\lbrack \lbrack \lbrack \mathcal {G}_k,  \mathcal W_1 \rbrack, \mathcal W_0\rbrack, \mathcal W_0 \rbrack  = 4 \lbrack \mathcal {G}_{k+1}, \mathcal W_0 \rbrack \qquad \qquad k\in \mathbb N^+, \label{eq:G6}
\\
& \lbrack \mathcal {G}_{k+1}, \mathcal {G}_{\ell+1} \rbrack = 0\qquad \qquad k,\ell \in \mathbb N. \label{eq:G7}
\end{align}
\end{lemma}
\begin{proof} Define a Lie algebra $\mathcal O^\vee$ by the presentation in the lemma statement. We show that $\mathcal O^\vee$ is isomorphic to $\mathcal O$.
By matrix multiplication, the matrices \eqref{eq:WWACE}--\eqref{eq:tGACE} satisfy the relations \eqref{eq:G1}--\eqref{eq:G7}.
Therefore, there exists a Lie algebra homomorphism $\natural : \mathcal O^\vee \to \mathcal O$ that sends
\begin{align*}
\mathcal W_0 \mapsto \mathcal W_0, \qquad \quad
\mathcal W_{1} \mapsto \mathcal W_{1}, \qquad \quad
\mathcal {G}_{k+1} \mapsto \mathcal {G}_{k+1} \qquad k \in \mathbb N.
\end{align*}
 We show that $\natural$ is a bijection. The map $\natural$ is surjective by Lemma \ref{lem:moreGen}. We show that $\natural$ is injective.
Let $O^\vee$ denote the Lie subalgebra of $\mathcal O^\vee$
generated by $\mathcal W_0, \mathcal W_1$. %%Let $G^\vee$ denote the subspace of $\mathcal O^\vee$
%%spanned by $\lbrace \mathcal {G}_{k+1} \rbrace_{k \in \mathbb N}$. 
%By \eqref{eq:G7} the subspace $G^\vee$
%s an abelian Lie subalgebra of $\mathcal O^\vee$. 
For $k \in \mathbb N$ we obtain    $\lbrack \mathcal W_0, \mathcal G_{k+1} \rbrack\in O^\vee$ and $\lbrack \mathcal W_1, \mathcal G_{k+1} \rbrack\in O^\vee$
using \eqref{eq:G3}--\eqref{eq:G6} and induction on $k$. By this and since $\mathcal W_0$, $\mathcal W_1$ generate $O^\vee$, we obtain
$\lbrack O^\vee, \mathcal G_{k+1} \rbrack\subseteq O^\vee$ for $k \in \mathbb N$.
Let $G^\vee$ denote the subspace of $\mathcal O^\vee$
spanned by $\lbrace \mathcal {G}_{k+1} \rbrace_{k \in \mathbb N}$. 
By the above comments $\lbrack O^\vee, G^\vee \rbrack\subseteq O^\vee$.  We also have $\lbrack G^\vee, G^\vee\rbrack =0$ by \eqref{eq:G7},
so $O^\vee + G^\vee$ is a Lie subalgebra of $\mathcal O^\vee$. But $O^\vee$ and $G^\vee$ together generate $\mathcal O^\vee$, 
 so $O^\vee+G^\vee=\mathcal O^\vee$.
Let $K$ denote the kernel of $\natural$.
Recall from Lemma  \ref{lem:OG}(ii)  the direct sum $\mathcal O = O + G$.
The map $\natural$ sends $O^\vee$ onto $O$. By Lemma  \ref{lem:OG}(i) the map $\natural$ sends $G^\vee$ onto $G$.
By these comments $K=K_1+ K_2$, where $K_1$ is the kernel for the restriction of $\natural$ to $O^\vee$, and
$K_2$ is the kernel for the restriction of $\natural$ to $G^\vee$.
We show that $K_1=0$.
 The elements $\mathcal W_0$, $\mathcal W_1$ of $\mathcal O^\vee$
satisfy the Dolan/Grady relations \eqref{eq:G1}, \eqref{eq:G2}.
 By this and Lemma \ref{def:OA}, there exists a Lie algebra homomorphism $\sharp: O \to \mathcal O^\vee$ that sends $W_0 \mapsto \mathcal W_0$
 and $W_1 \mapsto \mathcal W_1$. The $\sharp$-image of $O$ is equal to $O^\vee$. The composition $\natural\circ \sharp$ is injective, 
 so the restriction of $\natural$ to $O^\vee$ is injective. Therefore $K_1=0$.
 We show that $K_2=0$. The vectors $\lbrace \natural(\mathcal G_{k+1}) \rbrace_{k \in \mathbb N}$ form the basis for $G$ given in Lemma \ref{lem:OG}(i).
 Therefore the restriction of $\natural$ to $G^\vee$ is injective, so  $K_2=0$.
 By the above comments $K=0$, so $\natural$ is injective. We have shown that $\natural$ is a bijection, and therefore a Lie algebra isomorphism.
\end{proof}

\begin{lemma} \label{lem:Compact} The Lie algebra $\mathcal O$ has a presentation by generators $\mathcal W_0$, $\mathcal W_1$, $\lbrace \mathcal {\tilde G}_{k+1} \rbrace_{k\in \mathbb N}$ and
the following relations:
\begin{align*}
&\lbrack \mathcal W_0, \lbrack \mathcal W_0, \lbrack \mathcal W_0, \mathcal W_1 \rbrack \rbrack \rbrack = 4 \lbrack \mathcal W_0, \mathcal W_1 \rbrack, 
\\
&\lbrack \mathcal W_1, \lbrack \mathcal W_1, \lbrack \mathcal W_1, \mathcal W_0 \rbrack \rbrack \rbrack = 4 \lbrack \mathcal W_1, \mathcal W_0 \rbrack, 
\\
&\lbrack \mathcal W_0, \lbrack \mathcal W_0, \mathcal W_1 \rbrack \rbrack  =  \lbrack \mathcal W_0, \mathcal {\tilde G}_1 \rbrack, 
\\
&\lbrack \lbrack \mathcal W_0, \mathcal W_1\rbrack, \mathcal W_1 \rbrack   =  \lbrack \mathcal {\tilde G}_1, \mathcal W_1 \rbrack, 
\\
&\lbrack \mathcal W_0, \lbrack \mathcal W_0, \lbrack \mathcal W_1, \mathcal {\tilde G}_k \rbrack \rbrack \rbrack = 4 \lbrack \mathcal W_0, \mathcal {\tilde G}_{k+1} \rbrack  \qquad \qquad k\in \mathbb N^+, 
\\
&\lbrack \lbrack \lbrack \mathcal {\tilde G}_k,  \mathcal W_0 \rbrack, \mathcal W_1\rbrack, \mathcal W_1 \rbrack  = 4 \lbrack \mathcal {\tilde G}_{k+1}, \mathcal W_1 \rbrack  \qquad \qquad k\in \mathbb N^+, 
\\
& \lbrack \mathcal {\tilde G}_{k+1}, \mathcal {\tilde G}_{\ell+1} \rbrack = 0\qquad \qquad k,\ell \in \mathbb N. 
\end{align*}
\end{lemma}
\begin{proof} Apply the automorphism $\sigma$ to everything in Lemma  \ref{lem:Compacta}, and use Lemma \ref{lem:aaWG1}.
\end{proof}

\begin{lemma} The Lie algebra $\mathcal O$
has a presentation 
by generators
\begin{align}
\label{eq:4gens}
\lbrace \mathcal W_{-k}\rbrace_{k \in \mathbb N}, \quad 
\lbrace \mathcal W_{k+1}\rbrace_{k \in \mathbb N}, \quad
\lbrace \mathcal G_{k+1}\rbrace_{k \in \mathbb N}, \quad
 \lbrace \mathcal {\tilde G}_{k+1}\rbrace_{k \in \mathbb N}
\end{align}
 and the following relations. For $k, \ell \in \mathbb N$,
\begin{align}
&
 \lbrack \mathcal W_0, \mathcal W_{k+1}\rbrack= 
\lbrack \mathcal W_{-k}, \mathcal W_{1}\rbrack=
({\mathcal{\tilde G}}_{k+1} - \mathcal G_{k+1})/2,
\label{eq:3p1}
\\
&
\lbrack \mathcal W_0, \mathcal G_{k+1}\rbrack= 
\lbrack {\mathcal{\tilde G}}_{k+1}, \mathcal W_{0}\rbrack= 
4 \mathcal W_{-k-1}-4
 \mathcal W_{k+1},
\label{eq:3p2}
\\
&
\lbrack \mathcal G_{k+1}, \mathcal W_{1}\rbrack= 
\lbrack \mathcal W_{1}, {\mathcal {\tilde G}}_{k+1}\rbrack= 
4\mathcal W_{k+2}-4 
 \mathcal W_{-k},
\label{eq:3p3}
\\
&
\lbrack \mathcal W_{-k}, \mathcal W_{-\ell}\rbrack=0,  \qquad 
\lbrack \mathcal W_{k+1}, \mathcal W_{\ell+1}\rbrack= 0,
\label{eq:3p4}
\\
&
\lbrack \mathcal W_{-k}, \mathcal W_{\ell+1}\rbrack+
\lbrack \mathcal W_{k+1}, \mathcal W_{-\ell}\rbrack= 0,
\label{eq:3p5}
\\
&
\lbrack \mathcal W_{-k}, \mathcal G_{\ell+1}\rbrack+
\lbrack \mathcal G_{k+1}, \mathcal W_{-\ell}\rbrack= 0,
\label{eq:3p6}
\\
&
\lbrack \mathcal W_{-k}, {\mathcal {\tilde G}}_{\ell+1}\rbrack+
\lbrack {\mathcal {\tilde G}}_{k+1}, \mathcal W_{-\ell}\rbrack= 0,
\label{eq:3p7}
\\
&
\lbrack \mathcal W_{k+1}, \mathcal G_{\ell+1}\rbrack+
\lbrack \mathcal  G_{k+1}, \mathcal W_{\ell+1}\rbrack= 0,
\label{eq:3p8}
\\
&
\lbrack \mathcal W_{k+1}, {\mathcal {\tilde G}}_{\ell+1}\rbrack+
\lbrack {\mathcal {\tilde G}}_{k+1}, \mathcal W_{\ell+1}\rbrack= 0,
\label{eq:3p9}
\\
&
\lbrack \mathcal G_{k+1}, \mathcal G_{\ell+1}\rbrack=0,
\qquad 
\lbrack {\mathcal {\tilde G}}_{k+1}, {\mathcal {\tilde G}}_{\ell+1}\rbrack= 0,
\label{eq:3p10}
\\
&
\lbrack {\mathcal {\tilde G}}_{k+1}, \mathcal G_{\ell+1}\rbrack+
\lbrack \mathcal G_{k+1}, {\mathcal {\tilde G}}_{\ell+1}\rbrack= 0.
\label{eq:3p11}
\end{align}
\end{lemma} 
\begin{proof} Define a Lie algebra $\mathcal O^\vee$ by the presentation in the lemma statement. We show that $\mathcal O^\vee$ is isomorphic to $\mathcal O$.
By matrix multiplication, the matrices \eqref{eq:WWACE}--\eqref{eq:tGACE} satisfy
the relations \eqref{eq:3p1}--\eqref{eq:3p11}. Therefore, there exists a Lie algebra homomorphism $\natural : \mathcal O^\vee \to \mathcal O$ that sends
\begin{align*}
\mathcal W_{-k} \mapsto \mathcal W_{-k}, \qquad \quad
\mathcal W_{k+1} \mapsto \mathcal W_{k+1}, \qquad \quad
\mathcal G_{k+1} \mapsto \mathcal G_{k+1}, \qquad \quad
\mathcal {\tilde G}_{k+1} \mapsto \mathcal {\tilde G}_{k+1}
\end{align*}
for $k \in \mathbb N$. We show that $\natural$ is a bijection.  By Lemma \ref{lem:OA3} the map $\natural$ is surjective.  We show that $\natural $ is injective.
Let $O^\vee$ denote the Lie subalgebra of $\mathcal O^\vee$
generated by $\mathcal W_0, \mathcal W_1$.
For $k \in \mathbb N$ define $\mathcal Z^\vee_{k+1} \in \mathcal O^\vee$ by $\mathcal Z^\vee_{k+1} = (\mathcal G_{k+1} + \mathcal {\tilde G}_{k+1})/2$.
By the relation on the left in \eqref{eq:3p2}, we obtain $\lbrack \mathcal Z^\vee_{k+1}, \mathcal W_0 \rbrack=0$ for $k \in \mathbb N$.
By the relation on the left in \eqref{eq:3p3},  we obtain $\lbrack \mathcal Z^\vee_{k+1}, \mathcal W_1 \rbrack=0$ for $k \in \mathbb N$.
By the relations \eqref{eq:3p10}, \eqref{eq:3p11} we obtain $\lbrack \mathcal Z^\vee_{k+1}, \mathcal Z^\vee_{\ell+1} \rbrack=0$ for $k, \ell \in \mathbb N$.
Let $\mathcal Z^\vee$ denote the subspace of $\mathcal O^\vee$ spanned by $\lbrace \mathcal Z^\vee_{k+1}\rbrace_{k\in \mathbb N}$. By construction, the subspace 
$O^\vee + \mathcal Z^\vee$ is a Lie subalgebra of $\mathcal O^\vee$ whose center contains $\mathcal Z^\vee$.
 We now show that 
$O^\vee + \mathcal Z^\vee = \mathcal O^\vee$. To this end, we show that the generators \eqref{eq:4gens} of $\mathcal O^\vee$ are contained in 
$O^\vee + \mathcal Z^\vee$. This is established by induction on the ordering
\begin{align*}
\mathcal W_0, \quad \mathcal W_1, \quad \mathcal G_1, \quad \mathcal {\tilde G}_1, \quad 
\mathcal W_{-1}, \quad \mathcal W_2, \quad \mathcal G_2, \quad \mathcal {\tilde G}_2, \quad 
\mathcal W_{-2}, \quad \mathcal W_3, \quad 
\ldots
\end{align*} using the following equations, which come from \eqref{eq:3p1}--\eqref{eq:3p3}. For $n\geq 1$,
\begin{align*}
\mathcal G_n &=\mathcal Z^\vee_n + \lbrack \mathcal W_n, \mathcal W_0 \rbrack, &
\mathcal {\tilde G}_n &=\mathcal Z^\vee_n + \lbrack \mathcal W_0, \mathcal W_n \rbrack,
\\
\mathcal W_{-n} &= \mathcal W_n + \lbrack \mathcal {\tilde G}_n, \mathcal W_0 \rbrack/4, &
\mathcal W_{n+1} &= \mathcal W_{1-n} + \lbrack \mathcal W_1, \mathcal {\tilde G}_n \rbrack/4.
\end{align*}
\noindent We have shown that $O^\vee+ \mathcal Z^\vee = \mathcal O^\vee$. Consequently, $\mathcal Z^\vee_{k+1}$ is central in $\mathcal O^\vee$ for $k \in \mathbb N$.
We can now easily show that $\natural$ is injective. Let $K$ denote the kernel of $\natural$.
Recall from Lemma \ref{lem:ZZZ}(iii) the direct sum $\mathcal O = O + \mathcal Z$.
The map $\natural $ sends $O^\vee$ onto $O$. By Lemma  \ref{lem:ZZZACE}(ii), the map $\natural $ sends $\mathcal Z^\vee$ onto $\mathcal Z$.
By these comments $K=K_1+ K_2$, where $K_1$ is the kernel for the restriction of $\natural$ to $O^\vee$, and
$K_2$ is the kernel for the restriction of $\natural$ to $\mathcal Z^\vee$.
We show that $K_1=0$.
Setting $(k,\ell)=(0,1)$ in \eqref{eq:3p1}--\eqref{eq:3p4}, we find that the elements $\mathcal W_0$, $\mathcal W_1$ of $\mathcal O^\vee$
satisfy the Dolan/Grady relations.
 By this and Lemma \ref{def:OA}, there exists a Lie algebra homomorphism $\sharp: O \to \mathcal O^\vee$ that sends $W_0 \mapsto \mathcal W_0$
 and $W_1 \mapsto \mathcal W_1$. The $\sharp$-image of $O$ is equal to $O^\vee$. The composition $\natural \circ \sharp$ is injective, 
 so the restriction of $\natural$ to $O^\vee$ is injective. Therefore $K_1=0$.
 We show that $K_2=0$. The vectors $\lbrace \natural(\mathcal Z^\vee_{k+1}) \rbrace_{k \in \mathbb N}$ form the basis for $\mathcal Z$ given in Lemma \ref{lem:ZZZACE}(ii).
 Therefore the restriction of $\natural$ to $\mathcal Z^\vee$ is injective, so  $K_2=0$.
 By the above comments $K=0$, so $\natural$ is injective. We have shown that $\natural $ is a bijection, and therefore a Lie algebra isomorphism.
\end{proof}

\section{How $O$ and $\mathcal O$ are related to $O_q$ and $\mathcal O_q$}

In this section we describe how $O$ and $\mathcal O$ are related to the $q$-Onsager algebra $ O_q$ and its alternating central extension $\mathcal O_q$.
\medskip

\noindent Fix a nonzero $q \in \mathbb F$ that is not a root of unity. Recall the notation
\begin{align*}
\lbrack n \rbrack_q = \frac{q^n-q^{-n}}{q-q^{-1}} \qquad \qquad n \in \mathbb Z.
\end{align*}
For elements $X, Y$ in any algebra, define their
commutator and $q$-commutator by 
\begin{align*}
\lbrack X, Y \rbrack = XY-YX, \qquad \qquad
\lbrack X, Y \rbrack_q = q XY- q^{-1}YX.
\end{align*}
\noindent Note that 
\begin{align*}
\lbrack X, \lbrack X, \lbrack X, Y\rbrack_q \rbrack_{q^{-1}} \rbrack
= 
X^3Y-\lbrack 3\rbrack_q X^2YX+ 
\lbrack 3\rbrack_q XYX^2 -YX^3.
\end{align*}

\begin{definition} \label{def:U} \rm
(See \cite[Section~2]{bas1}, \cite[Definition~3.9]{qSerre}.)
%%\rm (See \cite[Corollary~3.2.6]{lusztig}.) 
Define the algebra $O_q$ by generators $W_0$, $W_1$ and relations
\begin{align}
\label{eq:qOns1}
&\lbrack \mathcal W_0, \lbrack W_0, \lbrack W_0, W_1\rbrack_q \rbrack_{q^{-1}} \rbrack =(q^2 - q^{-2})^2 \lbrack W_1, W_0 \rbrack,
\\
\label{eq:qOns2}
&\lbrack W_1, \lbrack W_1, \lbrack W_1, W_0\rbrack_q \rbrack_{q^{-1}}\rbrack = (q^2-q^{-2})^2 \lbrack W_0, W_1 \rbrack.
\end{align}
We call $O_q$ the {\it $q$-Onsager algebra}.
%The generators $W_0, W_1$ are called {\it standard}.
The relations \eqref{eq:qOns1}, \eqref{eq:qOns2}  are called the {\it $q$-Dolan/Grady relations}.
\end{definition}
\noindent We mention some symmetries of $O_q$. 

\begin{lemma}
\label{lem:aut} There exists an automorphism $\sigma$ of $O_q$ that sends $W_0 \leftrightarrow W_1$.
Moreover $\sigma^2 = {\rm id}$.
\end{lemma}

\begin{lemma}\label{lem:antiaut} {\rm (See \cite[Lemma~2.5]{z2z2z2}.)}
There exists an antiautomorphism $\dagger$ of $O_q$ that fixes each of $W_0$, $W_1$.
 Moreover $\dagger^2={\rm id}$.
\end{lemma}

\begin{lemma} {\rm (See \cite[Lemma~3.5]{compQons}.)} The maps $\sigma$, $\dagger$ commute.
\end{lemma}

\begin{remark}\label{rem:limit}\rm
We recover the Onsager algebra $O$ from $O_q$ by taking a limit $q\mapsto 1$. To keep things simple, assume that
$\mathbb F=\mathbb C$. In \eqref{eq:qOns1}, \eqref{eq:qOns2} make a change of variables $W_0 = \xi W'_0$ and $W_1 = \xi W'_1$ with $\xi = \sqrt{-1}(q-q^{-1})$. Simplify and set $q=1$ to obtain
the Dolan/Grady relations  \eqref{eq:DG}.
\end{remark}

\noindent
In Lemma \ref{def:ons} we gave a basis for $O$. We now review the analogous result for $O_q$.
In \cite[Theorem~4.5]{BK}, Baseilhac and Kolb obtain a PBW basis for $O_q$ that involves some elements
\begin{align}
\lbrace B_{n \delta+ \alpha_0} \rbrace_{n\in \mathbb N},
\qquad \quad 
\lbrace B_{n \delta+ \alpha_1} \rbrace_{n\in \mathbb N},
\qquad \quad 
\lbrace B_{n \delta} \rbrace_{n\in \mathbb N^+}.
\label{eq:Upbw}
\end{align}
These elements are recursively defined  as follows. Writing $B_\delta  = q^{-2}W_1 W_0 - W_0 W_1$ we have
\begin{align}
&B_{\alpha_0}=W_0,  \qquad \qquad 
B_{\delta+\alpha_0} = W_1 + 
\frac{q \lbrack B_{\delta}, W_0\rbrack}{(q-q^{-1})(q^2-q^{-2})},
\label{eq:line1a}
\\
&
B_{n \delta+\alpha_0} = B_{(n-2)\delta+\alpha_0}
+ 
\frac{q \lbrack B_{\delta}, B_{(n-1)\delta+\alpha_0}\rbrack}{(q-q^{-1})(q^2-q^{-2})} \qquad \qquad n\geq 2
\label{eq:line2a}
\end{align}
and 
\begin{align}
&B_{\alpha_1}=W_1,  \qquad \qquad 
B_{\delta+\alpha_1} = W_0 - 
\frac{q \lbrack B_{\delta}, W_1\rbrack}{(q-q^{-1})(q^2-q^{-2})},
\label{eq:line3a}
\\
&
B_{n \delta+\alpha_1} = B_{(n-2)\delta+\alpha_1}
- 
\frac{q \lbrack B_{\delta}, B_{(n-1)\delta+\alpha_1}\rbrack}{(q-q^{-1})(q^2-q^{-2})} \qquad \qquad n\geq 2.
\label{eq:line4a}
\end{align}
Moreover for $n\geq 2$,
\begin{equation}
\label{eq:Bdeltaa}
B_{n \delta} = 
q^{-2}  B_{(n-1)\delta+\alpha_1} W_0
- W_0 B_{(n-1)\delta+\alpha_1}  + 
(q^{-2}-1)\sum_{\ell=0}^{n-2} B_{\ell \delta+\alpha_1}
B_{(n-\ell-2) \delta+\alpha_1}.
\end{equation}
By \cite[Proposition~5.12]{BK} the elements $\lbrace B_{n\delta}\rbrace_{n\in \mathbb N^+}$ mutually commute. Additional relations involving the elements \eqref{eq:Upbw} can be found in \cite{BK}; see also
\cite[Section~2.2]{LuWang}.
\medskip

\noindent Referring to Remark \ref{rem:limit}, we now give the limiting values of the elements \eqref{eq:Upbw}. In  \eqref{eq:line1a}--\eqref{eq:Bdeltaa} and the expression for  $B_{\delta}$ below \eqref{eq:Upbw}, make a change of variables
\begin{align*}
B_{n\delta+\alpha_0} = \xi A_{-n}, \qquad \quad B_{n\delta+\alpha_1} = \xi A_{n+1},
\qquad \quad B_{m\delta} = 2\xi^2 B_m
\end{align*}
 for $n\geq 0 $ and $m\geq 1$. Simplify and set $q=1$ to obtain
 \begin{align*}
 \lbrack B_1, A_n \rbrack = A_{n+1} -A_{n-1}, \qquad \qquad 
 \lbrack A_m, A_0\rbrack = 2B_m 
 \end{align*}
 \noindent for $n \in \mathbb Z$ and $m\geq 1$.
 \noindent The elements $\lbrace A_k \rbrace_{k \in \mathbb Z}$, $\lbrace B_k \rbrace_{k \in \mathbb N^+}$ form the basis for $O$ given in Lemma \ref{def:ons}.
\medskip

\noindent In Definition
\ref{def:WWGG} we defined some elements
$\lbrace W_{-k}\rbrace_{k \in \mathbb N}$, $\lbrace W_{k+1}\rbrace_{k \in \mathbb N}$, $ \lbrace {\tilde G}_{k+1}\rbrace_{k \in \mathbb N}$
in $O$. These elements got used in 
 Lemma  \ref{lem:OA2} to obtain a basis for $O$, and in Lemma
\ref{lem:Olong} to obtain a presentation of $O$ by generators and relations. In \cite{conj} we investigated the analogous elements for $O_q$; our main result in this direction is
\cite[Conjecture~6.2]{conj}.
\medskip

\noindent By Lemma \ref{lem:centerO} the center of $O$ is zero. By 
\cite[Theorem~8.3]{kolb} 
the center of $O_q$ is $\mathbb F 1$.
\medskip

\noindent We now recall the algebra $\mathcal O_q$.

\begin{definition}\rm
\label{def:Aq}
(See 
\cite{BK05}, \cite[Definition~3.1]{basnc}.)
Define the algebra $\mathcal O_q$
by generators
\begin{align}
\label{eq:4gensa}
\lbrace \mathcal W_{-k}\rbrace_{n\in \mathbb N}, \qquad  \lbrace \mathcal  W_{k+1}\rbrace_{n\in \mathbb N},\qquad  
 \lbrace \mathcal G_{k+1}\rbrace_{n\in \mathbb N},
\qquad
\lbrace \mathcal {\tilde G}_{k+1}\rbrace_{n\in \mathbb N}
\end{align}
 and the following relations. For $k, \ell \in \mathbb N$,
\begin{align}
&
 \lbrack \mathcal W_0, \mathcal W_{k+1}\rbrack= 
\lbrack \mathcal W_{-k}, \mathcal W_{1}\rbrack=
({\mathcal{\tilde G}}_{k+1} - \mathcal G_{k+1})/(q+q^{-1}),
\label{eq:3p1a}
\\
&
\lbrack \mathcal W_0, \mathcal G_{k+1}\rbrack_q= 
\lbrack {\mathcal{\tilde G}}_{k+1}, \mathcal W_{0}\rbrack_q= 
\rho  \mathcal W_{-k-1}-\rho 
 \mathcal W_{k+1},
\label{eq:3p2a}
\\
&
\lbrack \mathcal G_{k+1}, \mathcal W_{1}\rbrack_q= 
\lbrack \mathcal W_{1}, {\mathcal {\tilde G}}_{k+1}\rbrack_q= 
\rho  \mathcal W_{k+2}-\rho 
 \mathcal W_{-k},
\label{eq:3p3a}
\\
&
\lbrack \mathcal W_{-k}, \mathcal W_{-\ell}\rbrack=0,  \qquad 
\lbrack \mathcal W_{k+1}, \mathcal W_{\ell+1}\rbrack= 0,
\label{eq:3p4a}
\\
&
\lbrack \mathcal W_{-k}, \mathcal W_{\ell+1}\rbrack+
\lbrack \mathcal W_{k+1}, \mathcal W_{-\ell}\rbrack= 0,
\label{eq:3p5a}
\\
&
\lbrack \mathcal W_{-k}, \mathcal G_{\ell+1}\rbrack+
\lbrack \mathcal G_{k+1}, \mathcal W_{-\ell}\rbrack= 0,
\label{eq:3p6a}
\\
&
\lbrack \mathcal W_{-k}, {\mathcal {\tilde G}}_{\ell+1}\rbrack+
\lbrack {\mathcal {\tilde G}}_{k+1}, \mathcal W_{-\ell}\rbrack= 0,
\label{eq:3p7a}
\\
&
\lbrack \mathcal W_{k+1}, \mathcal G_{\ell+1}\rbrack+
\lbrack \mathcal  G_{k+1}, \mathcal W_{\ell+1}\rbrack= 0,
\label{eq:3p8a}
\\
&
\lbrack \mathcal W_{k+1}, {\mathcal {\tilde G}}_{\ell+1}\rbrack+
\lbrack {\mathcal {\tilde G}}_{k+1}, \mathcal W_{\ell+1}\rbrack= 0,
\label{eq:3p9a}
\\
&
\lbrack \mathcal G_{k+1}, \mathcal G_{\ell+1}\rbrack=0,
\qquad 
\lbrack {\mathcal {\tilde G}}_{k+1}, {\mathcal {\tilde G}}_{\ell+1}\rbrack= 0,
\label{eq:3p10a}
\\
&
\lbrack {\mathcal {\tilde G}}_{k+1}, \mathcal G_{\ell+1}\rbrack+
\lbrack \mathcal G_{k+1}, {\mathcal {\tilde G}}_{\ell+1}\rbrack= 0.
\label{eq:3p11a}
\end{align}
In the above equations $\rho = -(q^2-q^{-2})^2$. Following \cite{pbwqO} we call $\mathcal O_q$ the {\it alternating central
extension of $O_q$}.
% The lgenerators 
%\eqref{eq:4gens} are called {\it alternating}.
%\noindent For notational convenience define
%\begin{align}
%{\mathcal G}_0 = -(q-q^{-1})\lbrack 2 \rbrack^2_q, \qquad \qquad 
%{\mathcal {\tilde G}}_0 = -(q-q^{-1}) \lbrack 2 \rbrack^2_q.
%\label{eq:GG0a}
%\end{align}
\end{definition}
\noindent Next we describe some symmetries of $\mathcal O_q$.
\begin{lemma}
\label{lem:autA} {\rm (See \cite[Remark~1]{basBel}.)} There exists an automorphism $\sigma$ of $\mathcal O_q$ that sends
\begin{align*}
\mathcal W_{-k} \mapsto \mathcal W_{k+1}, \qquad
\mathcal W_{k+1} \mapsto \mathcal W_{-k}, \qquad
\mathcal G_{k+1} \mapsto \mathcal {\tilde G}_{k+1}, \qquad
\mathcal {\tilde G}_{k+1} \mapsto \mathcal G_{k+1}
\end{align*}
%%%$\mathcal W_{-k} \leftrightarrow \mathcal W_{k+1}$ and
%%%%$\mathcal G_{k+1} \leftrightarrow \mathcal {\tilde G}_{k+1}$
 for $k \in \mathbb N$. Moreover $\sigma^2 = {\rm id}$.
\end{lemma}

\begin{lemma}\label{lem:antiautA} {\rm (See \cite[Lemma~3.7]{z2z2z2}.)} There exists an antiautomorphism $\dagger$ of $\mathcal O_q$ that sends
\begin{align*}
\mathcal W_{-k} \mapsto \mathcal W_{-k}, \qquad
\mathcal W_{k+1} \mapsto \mathcal W_{k+1}, \qquad
\mathcal G_{k+1} \mapsto \mathcal {\tilde G}_{k+1}, \qquad
\mathcal {\tilde G}_{k+1} \mapsto \mathcal G_{k+1}
\end{align*}
for $k \in \mathbb N$. Moreover $\dagger^2={\rm id}$.
\end{lemma}

\begin{lemma} \label{lem:sdcomA} {\rm (See \cite[Lemma~4.6]{compQons}.)} The maps $\sigma$, $\dagger $ commute.
\end{lemma}

\begin{remark}\label{rem:WWGGlim}\rm We recover $\mathcal O$ from $\mathcal O_q$ by taking a limit $q\mapsto 1$. Assume that
$\mathbb F=\mathbb C$, and define $\xi = \sqrt{-1} (q-q^{-1})$.
In \eqref{eq:3p1a}--\eqref{eq:3p11a}, make a change of variables
\begin{align*}
\mathcal W_{-k} =\xi \mathcal W'_{-k}, \qquad
\mathcal W_{k+1} =\xi \mathcal W'_{k+1}, \qquad
\mathcal G_{k+1} =\xi^2 \mathcal G'_{k+1}, \qquad
\mathcal {\tilde G}_{k+1} =\xi^2 \mathcal {\tilde G}'_{k+1}
\end{align*}
for $k \in \mathbb N$. Simplify and set $q=1$ to obtain the relations
\eqref{eq:3p1}--\eqref{eq:3p11}.
\end{remark}

\noindent We make some more comparisons between $O$, $\mathcal O$ and $O_q$, $\mathcal O_q$.
By Definition \ref{def:onsager}, $O$ is a certain Lie subalgebra of $L(\mathfrak{sl}_2)$. Analogously, the algebra $O_q$ is a
 %quantum symmetric pair 
 coideal subalgebra of the quantized enveloping algebra
$U_q(\widehat{\mathfrak{sl}}_2)$    \cite[Eqn. (1.2), Proposition~2.2]{bas8}; see also \cite[Example~7.6]{kolb}. %%see also  \cite{bas8, basXXZ,
By Definition \ref{def:OO}, $\mathcal O$ is a certain Lie subalgebra of $\mathcal L$, which in turn is
an ideal in $L(\mathfrak{gl}_2)$. Continuing the analogy, we expect that $\mathcal O_q$ is a coideal subalgebra of
$U_q(\widehat{\mathfrak{gl}}_2)$, but to date there are no results about this. However see \cite[Proposition~5.18]{FMA} for
a related result in which $\mathcal O_q$ is replaced by $\mathcal U^+_q$.
%%%the alternating central extension of the positive part of $U_q(\widehat{\mathfrak{sl}}_2)$.
\medskip

\noindent
In Lemma \ref{def:onsACE} we gave a basis for $\mathcal O$. It is an open problem to find an analogous PBW basis for $\mathcal O_q$.
\medskip

\noindent
Recall the center $\mathcal Z$ of $\mathcal O$.
In Lemma  \ref{lem:ZZZ}(iii) we obtained the direct sum $\mathcal O = O + \mathcal Z$. We now describe the analogous result for $\mathcal O_q$.
Let $\mathcal Z_q$ denote the center of $\mathcal O_q$. In \cite[Proposition~8.12]{pbwqO} we obtained some elements $\lbrace \mathcal Z_{k+1} \rbrace_{k\in \mathbb N}$ in $\mathcal Z_q$.
Let $\lbrace z_{k+1} \rbrace_{k\in \mathbb N}$ denote mutually commuting indeterminates. Let $\mathbb F \lbrack z_1, z_2, \ldots \rbrack$ denote
the algebra consisting of the polynomials in $z_1, z_2, \ldots $ that have all coefficients in $\mathbb F$. 
Let $\langle \mathcal W_0, \mathcal W_1\rangle $ denote the subalgebra of $\mathcal O_q$ generated by $\mathcal W_0$, $\mathcal W_1$.
\begin{lemma} \label{lem:sum}
{\rm (See \cite[Theorems~10.2--10.4]{pbwqO}.)} For the algebra $\mathcal O_q$ the following {\rm (i)--(iii)} hold:
\begin{enumerate}
\item[\rm (i)] 
there exists an algebra isomorphism
 $\mathbb F\lbrack z_1, z_2,\ldots \rbrack \to \mathcal Z_q$
 that sends $z_n \mapsto \mathcal Z_n$ for $n \geq 1$;
 \item[\rm (ii)] 
there exists an algebra isomorphism
$O_q \to \langle \mathcal W_0, \mathcal W_1\rangle$ 
that sends $W_0\mapsto \mathcal W_0$ and
$W_1\mapsto \mathcal W_1$;
\item[\rm (iii)] 
the multiplication map 
\begin{align*}
\langle \mathcal W_0, \mathcal W_1\rangle 
\otimes
\mathcal Z_q
 & \to   \mathcal O_q
\\
 w \otimes z  &\mapsto  wz
 \end{align*}
 is an isomorphism of algebras.
 \end{enumerate}
\end{lemma}
\noindent By Lemma \ref{lem:centerFix},
the automorphism $\sigma$ and the antiautomorphism $\dagger$ of $\mathcal O$ fix everything in $\mathcal Z$. By \cite[Lemma 8.10]{pbwqO},
the automorphism $\sigma$ and the antiautomorphism $\dagger$ of $\mathcal O_q$ fix everything in $\mathcal Z_q$.
\medskip

\noindent
By Lemma \ref{lem:Ogen}, the Lie algebra $\mathcal O$ is generated by $\mathcal W_0$, $\mathcal W_1$, $\mathcal Z$.
By  \cite[Corollary~8.23]{pbwqO}, the algebra $\mathcal O_q$ is generated by $\mathcal W_0$, $\mathcal W_1$, $\mathcal Z_q$.
\medskip

\noindent In Lemma  \ref{lem:OA3}, we gave a basis for $\mathcal O$ and the action of the Lie bracket on the basis.
We now describe the corresponding result for $\mathcal O_q$.
 In \cite[Theorem~6.1]{pbwqO} we showed that the generators 
 \eqref{eq:4gensa}
 give a PBW basis for $\mathcal O_q$. In 
 \cite[Proposition~5.1]{pbwqO} we gave the commutation relations satisfied by these generators.
 \medskip
 
\noindent In Lemma  \ref{lem:inj} we described an injective Lie algebra homomorphism $O \to \mathcal O$ that sends $W_0 \mapsto \mathcal W_0$
and $W_1 \mapsto \mathcal W_1$. The corresponding injective algebra homomorphism $O_q \to \mathcal O_q$ is mentioned in Lemma \ref{lem:sum}(ii) above.
In Lemma \ref{lem:surj} we described a surjective Lie algebra homomorphism $\rho: \mathcal O \to O$ that sends $\mathcal W_0 \mapsto W_0$ and $\mathcal W_1 \mapsto W_1$ and $\mathcal Z \mapsto 0$.
By Lemma \ref{lem:sum} there exists a surjective algebra homomorphism $\mathcal O_q \to O_q$ that sends $\mathcal W_0 \mapsto W_0$ and $\mathcal W_1 \mapsto W_1$ and $\mathcal Z_n \mapsto 0$ for $n \geq 1$.
\medskip

\noindent In Lemma  \ref{lem:OG} we gave two direct sum decompositions of $\mathcal O$. The analogous results for $\mathcal O_q$ can be found in
\cite[Proposition~12.2]{compQons} and
\cite[Theorem~1.2]{compQons}.
\medskip

\noindent In Lemma \ref{lem:Compacta} and  Lemma \ref{lem:Compact}, we gave two presentations of $\mathcal O$ by generators and relations. The analogous
presentations of $\mathcal O_q$ can be found in
\cite[Proposition~12.1]{compQons} and
\cite[Theorem~1.1]{compQons}.

\section{Acknowledgements}
The author thanks Pascal Baseilhac for giving this paper a close reading and offering valuable suggestions. 
In recent years, we have had many discussions about the Onsager Lie algebra and related topics.

%%%%%%%%%%%%%%%%%%%%%%%%%%%

%

\bigskip

\noindent Paul Terwilliger \hfil\break
\noindent Department of Mathematics \hfil\break
\noindent University of Wisconsin \hfil\break
\noindent 480 Lincoln Drive \hfil\break
\noindent Madison, WI 53706-1388 USA \hfil\break
\noindent email: {\tt terwilli@math.wisc.edu }\hfil\break

\end{document}